\documentclass[twocolumn]{autart}

\usepackage{umlaute,xspace,eurosym,verbatim}
\usepackage{amsmath,amssymb,bm,upgreek}
\usepackage{float,graphicx,epsfig,pstricks,pst-node,auto-pst-pdf,subfig}
\usepackage{ifpdf,cite}
\usepackage{multirow}
\usepackage{enumerate}
\usepackage{balance} 

\usepackage{xcolor,calc}

\psset{unit=1mm,framesep=10pt,fillstyle=none}  
\degrees[360]                                  
\psset{linewidth=0.9pt}                        
\providecommand{\lal}{\psset{linewidth=0.2pt}} 
\providecommand{\thl}{\psset{linewidth=0.3pt}} 
\providecommand{\mel}{\psset{linewidth=0.5pt}} 
\providecommand{\stl}{\psset{linewidth=0.9pt}} 
\providecommand{\bol}{\psset{linewidth=1.4pt}} 

\newtheorem{theorem}{Theorem}
\newtheorem{lemma}[theorem]{Lemma}

\newtheorem{remark}[theorem]{Remark}
\newtheorem{assumption}[theorem]{Assumption}

\newenvironment{proof}{\noindent\textit{Proof.}\rmfamily}{\hfill $\blacksquare$\vspace{2ex}}

\providecommand{\eg}{e.g.\@\xspace} %
\providecommand{\ie}{i.e.\@\xspace} %
\providecommand{\st}{\textnormal{s.t.}}
\providecommand{\pst}{\phantom{\textnormal{s.t.}}}
\providecommand{\ef}{\enspace.}
\providecommand{\ec}{\enspace,}

\providecommand{\fa}{\forall\,\,} %
\providecommand{\tp}{^{\textnormal{T}}} %
\providecommand{\to}{\rightarrow}

\DeclareMathOperator{\nrm}{nrm}
\DeclareMathOperator{\csd}{chi}
\DeclareMathOperator{\Tr}{Tr}
\providecommand{\Pb}{\textbf{P}}
\providecommand{\E}{\textbf{E}}

\providecommand{\lqr}{\textnormal{lqr}}
\providecommand{\cc}{\textnormal{cc}}

\providecommand{\epu}{\varepsilon_{u}}
\providecommand{\epx}{\varepsilon_{x}}
\providecommand{\bep}{\bar{\varepsilon}}
\providecommand{\bepu}{\bar{\varepsilon}_{u}}
\providecommand{\bepx}{\bar{\varepsilon}_{x}}
\providecommand{\tep}{\tilde{\varepsilon}}
\providecommand{\hx}{\hat{x}}
\providecommand{\hy}{\hat{y}}

\providecommand{\bS}{\bar{S}}
\providecommand{\bX}{\bar{X}}

\providecommand{\BN}{\mathbb{N}}
\providecommand{\BR}{\mathbb{R}}
\providecommand{\BRn}{\mathbb{R}^{n}}
\providecommand{\BRm}{\mathbb{R}^{m}}
\providecommand{\BRp}{\mathbb{R}^{p}}

\providecommand{\BU}{\mathbb{U}}
\providecommand{\BX}{\mathbb{X}}

\providecommand{\CN}{\mathcal{N}}

\begin{document}

\begin{frontmatter}
\runtitle{Linear Control of Chance Constrained Systems}  

\title{Linear Controller Design for Chance Constrained Systems\thanksref{footnoteinfo}}

\thanks[footnoteinfo]{This is an extended author's version of a work that was accepted for publication in Automatica. Changes resulting from the publishing process, such as peer review, editing, corrections, structural formatting, and other quality control mechanisms may not be reflected in this document. A definitive version is published in Automatica (volume 51, issue 1, 2015), doi: \texttt{10.1016/j.automatica.2014.10.096}.}

\author[SchiMor]{Georg Schildbach}\ead{schildbach@control.ee.ethz.ch},
\author[Goulart]{Paul Goulart}\ead{paul.goulart@eng.ox.ac.uk},
\author[SchiMor]{Manfred Morari}\ead{morari@control.ee.ethz.ch}

\address[SchiMor]{Automatic Control Laboratory, Swiss Federal Institute of Technology Zurich (Physikstrasse 3, 8092 Zurich, Switzerland)}
\address[Goulart]{Department of Engineering Science, University of Oxford (Parks Road, Oxford, OX1 3PJ, United Kingdom)}



\begin{abstract}
This paper is concerned with the design of a linear control law for linear systems with stationary additive disturbances. The objective is to find a state feedback gain that minimizes a quadratic stage cost function, while observing chance constraints on the input and/or the state. Unlike most of the previous literature, the chance constraints (and the stage cost) are not considered on each input/state of the transient response. Instead, they refer to the input/state of the closed-loop system in its stationary mode of operation. Hence the control is optimized for a long-run, rather than a finite-horizon operation.
The controller synthesis can be cast as a convex semi-definite program (SDP). The chance constraints appear as linear matrix inequalities. Both single chance constraints (SCCs) and joint chance constraints (JCCs) on the input and/or the state can be included. If the disturbance is Gaussian, additionally to WSS, this information can be used to improve the controller design. 
The presented approach can also be extended to the case of output feedback. The entire design procedure is flexible and easy to implement, as demonstrated on a short illustrative example.  
\end{abstract}

\end{frontmatter}

\section{Introduction}\label{Sec:Intro}

The problem of designing optimal linear controllers for unconstrained linear systems is well understood. The \emph{Linear Quadratic Gaussian} (LQG) regulator is a centerpiece of modern control theory \cite{AndMoore:2007}. However, incorporating constraints on the system state or input significantly complicates the mathematical design of regulators \cite{Bert2:2012,BertBrown:2007}.

A practical, ad-hoc approach to this problem consists of tuning the state and input cost weights appropriately. In some cases, however, this becomes a complicated and inefficient procedure, especially for systems of higher dimensions. Furthermore, the controller performance becomes suboptimal if the original cost function is physically meaningful. 

Different approaches have been proposed for constrained control design, including \emph{anti-windup control} \cite{Kothare:1994}, \emph{$\ell_1$ control} \cite{Dahleh:1995}, and \emph{set invariance theory} \cite{Bla:1999}. Recently, the approach of \emph{finite-horizon optimal control} (FHOC), e.g. \cite{BertBrown:2007,CinEtAl:2011,Primbs:2007}, has gained significant popularity, for two main reasons: first, the increasing availability of powerful computational hardware, and second, new advances in algorithms for stochastic optimization.

The basic idea of FHOC is to solve a \emph{multi-stage stochastic program} by numerical optimization online, repeatedly at each time step \cite{BertBrown:2007,CinEtAl:2011,Primbs:2007}. This stochastic optimization approach offers great flexibility for handling various \emph{cost functions} and \emph{chance constraints} on the states and inputs. In particular, recently proposed methods have improved the handling of joint chance constraints of various types \cite{BertSim:2006,ChenEtAl:2010}, and the computation of distributionally robust solutions \cite{CalaGhao:2006}. Nonetheless, FHOC provides optimal control of the system only over a finite time horizon. Or, if applied in a receding horizon fashion, it fails to give any guarantees about constraint satisfaction, optimality, or sometimes even feasibility.

This paper presents a new approach that differs from the above. Its goal is to design a linear control law for a linear time-invariant system that has additive stochastic disturbances. Unlike FHOC, the presented method does not solve a multi-stage stochastic program for optimizing the \emph{transient response} of the system. Instead, the method is concerned with the \emph{stationary regime}, i.e., the stationary distributions of the state and the input. These distributions can be shaped by the decision variable (i.e., the linear feedback gain) according to a quadratic cost function and chance constraints. 

The focus on the stationary regime means that, in contrast to FHOC, the presented method entirely ignores the system's initial condition. Therefore the cost is optimized and the chance constraints are satisfied \emph{asymptotically in time}, rather than point-wise in each time step \cite{MeynTweedie:2009}. Unlike for FHOC, this implies that the state feedback law can be defined over the entire state space. As a consequence, the method is able to handle disturbances with possibly unbounded support. Moreover, the computed feedback law is linear and therefore easy to implement on a real system.

The mathematical problem for designing a linear controller for the stationary regime can formulated as a \emph{semidefinite program} (SDP). The stationary joint distributions of the state and the input (more precisely, only their means and their variances) become additional decision variables in the SDP. They vary in correspondence with the primary decision variable, the linear feedback gain. The quadratic cost function as well as a given set of single and/or joint chance constraints on the state and/or the inputs can be formulated using \emph{linear matrix inequalities} (LMIs) on the decision variables.

The presented approach covers two different cases for the additive disturbances: \emph{wide-sense stationary} (WSS) processes with limited moment information and \emph{Gaussian} (NRM) processes. For WSS processes with known mean and variance, the satisfaction of the chance constraints is \emph{robust} with respect to all possible stationary distributions consistent with the given moment information; cf.\ \cite{LagBar:2006,ZymlerEtAl:2013}. This design is generally conservative, however, because the worst-case distribution of all possible stationary distributions is generally not achieved. For NRM processes, it is shown that this conservatism can be eliminated.

This paper extends the work of Zhou and Cogill \cite{ZhouCog:2013} by incorporating input constraints, improving their bounds for single state constraints, considering also Gaussian disturbances, and covering also output feedback. The content of this paper is closely related also to the work of van Parys et al.\ \cite{VanParEtAl:2014}. Their primary focus however lies on constraining the \emph{conditional value-at-risk}, treating it as a conservative approximation to chance constraints. Moreover, this paper is additionally concerned with reducing conservatism by including single chance constraints and Gaussian disturbances as special cases.

\section{Problem Description}\label{Sec:Problem}


Consider the discrete time, linear time-invariant system
\begin{subequations}\label{Equ:System}\begin{align}
    &x_{t+1}=Ax_t+Bu_t+w_t\ec\\
    &y_t=Cx_t+v_t\ec
\end{align}\end{subequations}
together with a fixed initial condition $x_{0}\in\BRn$. The variables $x_{t}\in\BRn$, $u_{t}\in\BRm$, $y_{t}\in\BRp$ are used to denote the \emph{state}, the \emph{input}, the \emph{output} at time $t\in\BN$, while $w_{t}\in\BRn$ and $v_{t}\in\BRm$ represent random disturbances and measurement noise, respectively. The following further assumptions hold throughout; cf.\ \cite{Ludyk1:1995,Ludyk2:1995}.

\begin{assumption}[Control System]\label{Ass:System}
    (a) The matrix pair $(A,B)$ is stabilizable.
    (b) The matrix pair $(A,C)$ is observable.
    (c) The disturbances $\{w_{t}\}_{t\in\BN}$ and the measurement noise $\{v_{t}\}_{t\in\BN}$ are
    stationary white noise processes.
    (d) All disturbances have zero mean $\E[w_{t}]=0$ and the same finite covariance
    $W:=\E[w_{t}w\tp_{t}]$, which is positive definite ($W\succ 0$).
    (e) All measurement noise has zero mean $\E[v_{t}]=0$ and the same finite covariance
    $V:=\E[v_{t}v\tp_{t}]$, which is positive definite ($V\succ 0$).
\end{assumption}

Assumption \ref{Ass:System} ensures that $\{w_{t}\}_{t\in\BN}$ and $\{v_{t}\}_{t\in\BN}$ are \emph{wide-sense stationary} (WSS) processes \cite{Simon:2006}. In particular, their terms need not be identically distributed between time steps; only their first two moments are assumed identical and known. However, the important problem where both $\{w_{t}\}_{t\in\BN}$ and $\{v_{t}\}_{t\in\BN}$ follow \emph{Gaussian processes} (NRM), i.e.\ all terms are identically normally distributed, is considered as an special case. It will allow for sharper bounds to be obtained.

\subsection{Stationary State and Input}

Initially, assume that the state of \eqref{Equ:System} is measurable for the purpose of state feedback. The treatment of output feedback problems is deferred until Section \ref{Sec:OutputFeedback}. A basic control objective is to design a \emph{linear feedback gain} $K\in\BR^{m\times n}$ such that $u_{t}=Kx_{t}$ stabilizes the system \eqref{Equ:System} and regulates its state around the origin.

Due to the random disturbances, the closed-loop states defined by
\begin{equation}\label{Equ:ClosedLoop}
    x_{t+1}=\bigl(A+BK\bigr)x_{t}+w_{t}
\end{equation}
constitute a sequence of random variables:
\begin{equation}\label{Equ:RandStates}
    x_{t}\bigl(x_{0},w_{0},w_{1},\dots,w_{t-1}\bigr)\qquad\fa t=1,2,\dots\ef
\end{equation}
If the disturbances $\{w_{t}\}_{t\in\BN}$ are WSS (or NRM) and $K$ is stabilizing, then the closed-loop states $\{x_{t}\}_{t\in\BN}$ are also, asymptotically, WSS (or NRM); see\ \cite[Sec.\,9.2]{PapPil:2002}, \cite[Thm.\,17.6.2]{MeynTweedie:2009}. Indeed, the closed loop states converge to a zero mean and a \emph{stationary covariance} $\bX\in\BR^{n\times n}$ as $t\to\infty$, as shown in Lemma \ref{The:StationaryVar} below.

\begin{lemma}[Stationary Variance]\label{The:StationaryVar}
	Let $K$ be strictly stabilizing for $(A,B)$ and $\{w_{t}\}_{t\in\BN}$ be WSS or NRM. The
	covariance of the closed-loop states $\{x_{t}\}_{t\in\BN}$ converge to a unique stationary 
	value $\bX\succ 0$ that satisfies the discrete time Lyapunov equation
    \begin{equation}\label{Equ:StationaryVar}
    	\bX - (A+BK)\bX(A+BK)\tp - W = 0\ef
    \end{equation}
\end{lemma}
	
\begin{proof}
	For any $t\in\BN$, let $X_{t}$ denote the variance of~$x_{t}$. The sequence of variances
	$\{X_{t}\}_{t\in\BN}$ satisfies the discrete time Lyapunov equation
	\begin{align}\label{Equ:StationaryEq1}
		\E\bigl[x_{t+1}&x_{t+1}\tp\bigr]= & \nonumber\\
		&=\E\bigl[\bigl((A+BK)x_t+w_t\bigr)\bigl((A+BK)x_t+w_t\bigr)\tp\bigr]
		                                       \nonumber\\
		&=\E\bigl[(A+BK)x_tx_t\tp(A+BK)\tp\bigr]+\E\bigl[w_tw_t\tp\bigr]\nonumber\\
		\Longrightarrow\; &X_{t+1}=(A+BK)X_t(A+BK)\tp+W\ef
	\end{align}
	Since $A+BK$ is strictly stable, \eqref{Equ:StationaryEq1} has a unique fixed point $\bX$
	characterized by \eqref{Equ:StationaryVar} \cite{CaiMay:1970}.
	By virtue of Assumption \ref{Ass:System}(d), $\bX \succeq W \succ 0$. 
\end{proof}

\begin{remark}[Notation]\label{Rem:Notation}
	(a) The stationary covariance $\bX$ is occasionally denoted $\bX(K)$, when the dependency on the
	controller $K$ should be emphasized.
	(b) Let $x$ (without time index) be a random variable with mean $0$ and covariance $\bX$. With
	some abuse of terminology, $x$ shall be called the \emph{stationary state} of system
	\eqref{Equ:ClosedLoop}; see \cite{MeynTweedie:2009}.
	(c) Let $u:=Kx$ (without time index) be the according \emph{stationary input} of system
	\eqref{Equ:ClosedLoop}. Hence $u$ a random variable with mean $0$ and covariance $K\bX K\tp$.
\end{remark}

If the disturbances are WSS, the distributions of $x$ and $u$ are not fixed; only their first two moments are fixed. If the disturbances are NRM, both $x$ and $u$ follow a normal distribution; cf.\ \cite[Thm.\,17.6.2]{MeynTweedie:2009}.

\subsection{Optimal Controller Design}\label{Sec:OptControl}

A \emph{stabilizing} feedback gain $K$ is said to be \emph{feasible} if it satisfies a given set of chance constraint(s) on the stationary state and/or input:
\begin{subequations}\label{Equ:ChanceConstr}\begin{align}
	&\Pb\bigl[x\in\BX\bigr]\geq 1-\epx\ec\label{Equ:StateCon}\\
	&\Pb\bigl[u\in\BU\bigr]\geq 1-\epu\ef\label{Equ:InputCon}
\end{align}\end{subequations}
Here $\BX\subset\BRn$ and $\BU\subset\BRm$ are non-empty polytopic or ellipsoidal sets containing the origin, and $\epx\in(0,1)$ and $\epu\in(0,1)$ indicate desired maximum levels of the constraint violation probability.

A \emph{feasible} feedback gain $K$ is said to be \emph{optimal} if it minimizes the quadratic cost function
\begin{equation}\label{Equ:CostFunction}
	\E\bigl[x\tp Qx+u\tp Ru\bigr]
\end{equation}
among all feasible feedback gains. Here $Q\in\BR^{n\times n}$ and $R\in\BR^{m\times m}$ are positive definite weighting matrices.

Note that the constraints \eqref{Equ:ChanceConstr} and the cost function \eqref{Equ:CostFunction} involve the stationary state $x$ and input $u$, which are independent of the initial condition $x_{0}$. This differs from the standard LQR control problem, where the infinite sum of (deterministic) stage costs for the transient response of \eqref{Equ:ClosedLoop} is minimized; cf.\ \cite{AndMoore:2007}. By invoking a large number law, the chance constraints are hence satisfied and the cost function is minimized \emph{asymptotically in time}; see \cite[Cha.\,17]{MeynTweedie:2009}.

\subsection{Paper Outline}\label{Sec:Content}

The problem of finding an optimal feedback gain, as described in Section \ref{Sec:OptControl}, is referred to as the \emph{controller synthesis problem} (CSP). It will be shown that the CSP can be formulated as a \emph{semidefinite program} (SDP). It solves, simultaneously, for the feedback gain $K$ and the stationary variance $\bX(K)$. 

To this end, the stationarity condition (Section \ref{Sec:Stationarity}), single chance constraints (Section \ref{Sec:SCC}), joint chance constraints (Section \ref{Sec:JCC}), and finally the objective function (Section \ref{Sec:Synthesis}) are reformulated as \emph{linear matrix inequalities} (LMIs). Then the approach is extended to the case of output feedback (Section \ref{Sec:OutputFeedback}) and finally illustrated on a short numerical example (Section \ref{Sec:Example}).

\section{Stationary Distribution}\label{Sec:Stationarity}

The following lemma characterizes the set of all stabilizing feedback gains $K$ and their associated stationary covariances matrices $\bX(K)$.  

\begin{lemma}[Stationarity Condition]\label{The:Stationarity} If $K\in\BR^{m\times n}$ and 
	$X\in\BR^{n\times n}$ are chosen such that
    \begin{equation}\label{Equ:Stationarity}
    	\begin{bmatrix} X-W & (A+BK)X \\ X\tp(A+BK)\tp & X \end{bmatrix} \succeq 0 \ec
    \end{equation}
    then $(A+BK)$ is strictly stable and $X \succeq \bX(K)$.
\end{lemma}

\begin{proof}
	Since $W\succ 0$ by assumption, the upper left-hand term in \ref{Equ:Stationarity} ensures that
	$X\succ 0$.
	By a transformation using Schur complements \cite[Sec.\,2.1]{BoydEtAl:1994}, 
	\eqref{Equ:Stationarity} is equivalent to
	\begin{gather}\label{Equ:StationaryIneq}
		\begin{bmatrix} X-W & (AX+BKX) \\ (AX+BKX)\tp & X \end{bmatrix} \succeq 0\nonumber\\
		\Updownarrow\nonumber\\
		\begin{bmatrix} X-W & (A+BK) \\ (A+BK)\tp & X^{-1} \end{bmatrix} \succeq 0\nonumber\\
        \Updownarrow\nonumber\\
        X - (A+BK)X(A+BK)\tp - W \succeq 0\ef
	\end{gather}
	Since both $X$ and $W$ are positive definite, this constitutes a discrete time Lyapunov
	equation, so $(A+BK)$ is strictly stable. 
	Define a matrix $\tilde{X}\in\BR^{n\times n}$ such that $X=\bX(K)+\tilde{X}$, and substitute
	this into \eqref{Equ:StationaryIneq}. Since $\bX(K)$ satisfies \eqref{Equ:StationaryVar} by
	Lemma \ref{Equ:StationaryVar},
	\begin{equation*}
		\tilde{X} - (A+BK)\tilde{X}(A+BK) \succeq 0\ef
	\end{equation*}
	Because $(A+BK)$ is stable, $\tilde{X}$ is itself a solution to a discrete time Lyapunov
	equation,  so $\tilde{X}\succeq 0$ and $X\succeq \bX(K)$.
\end{proof}

\begin{remark}[Partial Ordering]
	(a) The stationary covariance $\bX(K)$ is the \emph{minimal matrix} (according to the
	semidefinite partial ordering) that satisfies the matrix inequality \eqref{Equ:Stationarity}. 
	(b) Therefore, the stationary covariance $\bX(K)$ can be computed directly for a given feedback
	gain $K$, \eg by minimizing $\Tr(X)$ subject to \eqref{Equ:Stationarity}. It shall be seen later 
	in the paper that, in fact, $X=\bX(K)$ for the CSP as well.
\end{remark}

Observe that the matrix inequality \eqref{Equ:Stationarity} contains the bilinear term $KX$. Since both $K$ and $X$ are decision variables of the CSP, the invertible change of variables
\begin{equation}\label{Equ:Transform}
	Y:=KX \qquad\Longrightarrow\qquad K=YX^{-1}
\end{equation}
leads to a linear matrix inequality in $Y$ and $X$.

\section{Single Chance Constraints (SCCs)}\label{Sec:SCC}

Let $g\in\BRn$ and $h\in\BR$ be a vector and a scalar, respectively. An SCC on the state $x$ can be either \emph{one-sided}, where $\BX$ is a half space, or \emph{two-sided}, where $\BX$ is the parallel intersection of two half spaces:
\begin{equation}\label{Equ:StateSCC}
	\Pb\bigl[g\tp x\leq h\bigr]\geq 1-\epx\;\;\text{or}\;\;\Pb\bigl[|g\tp x|\leq h\bigr]
	\geq 1-\epx\,.
\end{equation}
Similarly, for a vector $f\in\BRn$ and $e\in\BR$, an SCC on the input can be either \emph{one-sided}, where $\BU$ is a half space, or \emph{two-sided}, where $\BU$ is the symmetric intersection of two half spaces:
\begin{equation}\label{Equ:InputSCC}
	\Pb\bigl[f\tp u\leq e\bigr]\geq 1-\epu\;\;\text{or}\;\;\Pb\bigl[|f\tp u|\leq e\bigr]
	\geq 1-\epu\,.
\end{equation}
In the remainder of this section, the reformulation of \eqref{Equ:StateSCC} and \eqref{Equ:InputSCC} into LMIs is described.

Recall from Remark \ref{Rem:Notation} that the stationary state $x$ and the stationary input $u$ are a random variables with zero means and covariance $\bX$ and $K\bX K\tp$, respectively. Hence $g\tp x$, $f\tp Kx$ are a scalar random variables with zero means and variances $g\tp \bX g$, $f\tp K\bX K\tp f$. Viewed as such, \eqref{Equ:StateSCC} and \eqref{Equ:InputSCC} become (one-sided or two-sided) tail bounds on the distributions of $g\tp x$ and $f\tp Kx$. 

If the disturbance is WSS and the first two moments of $x$, $u$ are available, this tail bound can be obtained from Chebyshev's inequality; cf.\ \cite[Sec.\,3]{MitzUpf:2005}. If the disturbance is NRM, the tail bound can be made exact by inverting the cumulative \emph{standard normal distribution function}, denoted $\nrm:\BR\to[0,1]$; cf.\ \cite[Equ.\,26.2.29]{Abramowitz:1970}. This inverse is easy to compute numerically since $\nrm(\cdot)$ is monotonic.

\begin{lemma}[SCC on the State]\label{The:SCCState}
	Let $X$ be any matrix such that $X\succeq \bX$. Then the one-sided or two-sided SCC on the state
	\eqref{Equ:StateSCC} with a WSS or NRM disturbance is satisfied if
	\begin{equation}\label{Equ:SCCState}
	    g\tp Xg\leq \alpha h^2\ec
	\end{equation}
	where the appropriate value of $\alpha$ is given in Table \ref{Tab:SCC}.
\end{lemma}

\begin{proof}
	First, consider the case of a WSS disturbance. For the two-sided SCC, Chebyshev's inequality
	\cite[Thm.\,3.6]{MitzUpf:2005} provides that 
	\begin{equation*}
		\Pb\bigl[|g\tp x|> h\bigr]\leq\frac{g\tp \bX g}{h^2}\leq\frac{g\tp Xg}{h^2}\ec
	\end{equation*}
	since $g\tp \bX g \le g\tp Xg$ for any $g$ by assumption. 
	Hence requiring that the right-hand side be less than $\epx$ is sufficient to satisfy the chance
	constraint, proving \eqref{Equ:SCCState} with $\alpha=\epx$. The condition with $\alpha=2\epx$
	for a one-sided SCC follows by symmetry.
	
	Second, consider the case of a NRM disturbance. Since the random variable $g\tp x$ is normally
	distributed with mean $0$ and variance $g\tp \bX g$, it holds that $g\tp x/\sqrt{g\tp \bX g}\sim
	\CN(0,1)$. Hence for the one-sided SCC
    \begin{align*}
		\Pb\bigl[g\tp x>h\bigr]&=1-\nrm\bigl(h/\sqrt{g\tp \bX g}\bigr)\\
							   &\le1-\nrm(\bigl(h/\sqrt{g\tp  Xg}\bigr)\ec 
	\end{align*}
	where the inequality follows since $\nrm(\cdot)$ is monotone.  
	Requiring that the right-hand side be less than $\epx$ yields the corresponding value for 
	$\alpha$ in Table \ref{Tab:SCC}, after some algebraic manipulations. The case for a 
	two-sided SCC follows by symmetry.
\end{proof}

\begin{lemma}[SCC on the Input]\label{The:SCCInput}
	Let $X$ be any matrix such that $X\succeq \bX$. Then the one-sided or two-sided SCC on the input 
	\eqref{Equ:InputSCC} with a WSS or NRM disturbance is satisfied if
	\begin{equation}\label{Equ:SCCInput}
	    \begin{bmatrix} \beta e^2 & f\tp Y \\ Y\tp f & X \end{bmatrix} \succeq 0\ec
	\end{equation}
	where the appropriate value of $\beta$ is given in Table \ref{Tab:SCC}.
\end{lemma}

\begin{proof}
	Using an argument identical to that in the proof of Lemma \ref{The:SCCState}, the one-sided or
	two-sided SCC on the input with a WSS or NRM disturbance is satisfied if
	\begin{equation}\label{Equ:TheSCCInputEqu}
	    f\tp K \bX K\tp \le f\tp K XK\tp f\leq \beta e^2.
	\end{equation}
	Inequality \eqref{Equ:TheSCCInputEqu} is equivalent to \eqref{Equ:SCCInput} by \cite[Sec.\,2.1]
	{BoydEtAl:1994}, using the variable transformation \eqref{Equ:Transform}.
\end{proof}

\begin{table}[H]
    \renewcommand\arraystretch{1.3}
    \centering
    \begin{tabular}{l|ll}
        \multicolumn{1}{c|}{\textbf{Case}} & \multicolumn{2}{c}{\textbf{Value}}\\
        \hline\hline
        WSS, one-sided & $\alpha=2\epx\ec$ & $\beta=2\epu$\\
        WSS, two-sided & $\alpha=\epx\ec$ &  $\beta=\epu$\\
        NRM, one-sided & \multicolumn{2}{l}{$\alpha=\bigl(\nrm^{-1}\bigl(1-\epx\bigr)\bigr)^{-2}$
                         \ec}\\
                       & \multicolumn{2}{l}{$\beta=\bigl(\nrm^{-1}\bigl(1-\epu\bigr)\bigr)^{-2}$}\\
        NRM, two-sided & \multicolumn{2}{l}{$\alpha=\bigl(\nrm^{-1}\bigl(1-\epx/2\bigr)\bigr)^{-2}$
                         \ec}\\
                       & \multicolumn{2}{l}{$\beta=\bigl(\nrm^{-1}\bigl(1-\epu/2\bigr)\bigr)^{-2}$}
    \end{tabular}
    \caption{Appropriate values of $\alpha$, $\beta$ in \eqref{Equ:SCCState}, \eqref{Equ:SCCInput},
    if the disturbances are wide-sense stationary (WSS) or Gaussian (NRM) and the chance
    constraint is either one-sided or two-sided.\label{Tab:SCC}}
    \renewcommand\arraystretch{1.0}
\end{table}

Figure \ref{Fig:SCCBounds} compares the different bounds in Table \ref{Tab:SCC}. Clearly, the two-sided bound is more restrictive on $X$ than the one-sided bound for the same probability level $\epx$. Furthermore, the condition on $X$ is less restrictive if the disturbances are assumed to be NRM rather than general WSS. This difference becomes larger for higher probability levels $\epx$.

\begin{figure}[H]
    \begin{center}
	\begin{pspicture}(2,-10)(85,40)
		\footnotesize
	    \rput[tc](20,-8){\textit{(a) One-sided SCC.}}
	    \rput[tc](65,-8){\textit{(b) Two-sided SCC.}}
	    \scriptsize
        \thl
        \psline[arrowsize=4.5pt]{->}(3,0)(39,0)
        \psline[arrowsize=4.5pt]{->}(4,-1)(4,35)
        \psline[arrowsize=4.5pt]{->}(47,0)(83,0)
        \psline[arrowsize=4.5pt]{->}(48,-1)(48,35)
        \rput[tc](39,-2.5){$\epx$}
        \rput[mr](2,34){$\alpha$}
        \rput[tc](83,-2.5){$\epx$}
        \rput[mr](46,34){$\alpha$}
        \psline{-}(14,-0.5)(14,0.5)\psline{-}(24,-0.5)(24,0.5)\psline{-}(34,-0.5)(34,0.5)
        \rput[tc](14,-3){$0.1$}\rput[tc](24,-3){$0.2$}\rput[tc](34,-3){$0.3$}
        \psline{-}(3.5,15)(4.5,15)\psline{-}(3.5,30)(4.5,30)
        \rput[mr](2.5,15){$2$}\rput[mr](2.5,30){$4$}
        \stl
        \psline[showpoints=false,linestyle=solid](4.00,0.00)(4.10,0.01)(4.20,0.03)(4.30,0.04)(4.40,0.06)(4.50,0.07)(4.60,0.09)(4.70,0.10)(4.80,0.12)(4.90,0.13)(5.00,0.15)(5.10,0.16)(5.20,0.18)(5.30,0.19)(5.40,0.21)(5.50,0.22)(5.60,0.24)(5.70,0.25)(5.80,0.27)(5.90,0.28)(6.00,0.30)(6.10,0.31)(6.20,0.33)(6.30,0.34)(6.40,0.36)(6.50,0.38)(6.60,0.39)(6.70,0.40)(6.80,0.42)(6.90,0.43)(7.00,0.45)(7.10,0.46)(7.20,0.48)(7.30,0.49)(7.40,0.51)(7.50,0.53)(7.60,0.54)(7.70,0.55)(7.80,0.57)(7.90,0.58)(8.00,0.60)(8.10,0.61)(8.20,0.63)(8.30,0.65)(8.40,0.66)(8.50,0.67)(8.60,0.69)(8.70,0.70)(8.80,0.72)(8.90,0.73)(9.00,0.75)(9.10,0.76)(9.20,0.78)(9.30,0.79)(9.40,0.81)(9.50,0.82)(9.60,0.84)(9.70,0.85)(9.80,0.87)(9.90,0.89)(10.00,0.90)(10.10,0.92)(10.20,0.93)(10.30,0.95)(10.40,0.96)(10.50,0.98)(10.60,0.99)(10.70,1.00)(10.80,1.02)(10.90,1.03)(11.00,1.05)(11.10,1.07)(11.20,1.08)(11.30,1.09)(11.40,1.11)(11.50,1.12)(11.60,1.14)(11.70,1.16)(11.80,1.17)(11.90,1.19)(12.00,1.20)(12.10,1.22)(12.20,1.23)(12.30,1.24)(12.40,1.26)(12.50,1.28)(12.60,1.29)(12.70,1.30)(12.80,1.32)(12.90,1.33)(13.00,1.35)(13.10,1.36)(13.20,1.38)(13.30,1.40)(13.40,1.41)(13.50,1.43)(13.60,1.44)(13.70,1.45)(13.80,1.47)(13.90,1.48)(14.00,1.50)(14.10,1.51)(14.20,1.53)(14.30,1.54)(14.40,1.56)(14.50,1.57)(14.60,1.59)(14.70,1.60)(14.80,1.62)(14.90,1.63)(15.00,1.65)(15.10,1.66)(15.20,1.68)(15.30,1.69)(15.40,1.71)(15.50,1.73)(15.60,1.74)(15.70,1.76)(15.80,1.77)(15.90,1.78)(16.00,1.80)(16.10,1.81)(16.20,1.83)(16.30,1.84)(16.40,1.86)(16.50,1.88)(16.60,1.89)(16.70,1.91)(16.80,1.92)(16.90,1.93)(17.00,1.95)(17.10,1.96)(17.20,1.98)(17.30,2.00)(17.40,2.01)(17.50,2.03)(17.60,2.04)(17.70,2.06)(17.80,2.07)(17.90,2.08)(18.00,2.10)(18.10,2.11)(18.20,2.13)(18.30,2.15)(18.40,2.16)(18.50,2.17)(18.60,2.19)(18.70,2.21)(18.80,2.22)(18.90,2.23)(19.00,2.25)(19.10,2.27)(19.20,2.28)(19.30,2.29)(19.40,2.31)(19.50,2.33)(19.60,2.34)(19.70,2.35)(19.80,2.37)(19.90,2.38)(20.00,2.40)(20.10,2.42)(20.20,2.43)(20.30,2.45)(20.40,2.46)(20.50,2.48)(20.60,2.49)(20.70,2.51)(20.80,2.52)(20.90,2.53)(21.00,2.55)(21.10,2.56)(21.20,2.58)(21.30,2.59)(21.40,2.61)(21.50,2.62)(21.60,2.64)(21.70,2.66)(21.80,2.67)(21.90,2.69)(22.00,2.70)(22.10,2.71)(22.20,2.73)(22.30,2.75)(22.40,2.76)(22.50,2.77)(22.60,2.79)(22.70,2.81)(22.80,2.82)(22.90,2.83)(23.00,2.85)(23.10,2.87)(23.20,2.88)(23.30,2.90)(23.40,2.91)(23.50,2.93)(23.60,2.94)(23.70,2.95)(23.80,2.97)(23.90,2.98)(24.00,3.00)(24.10,3.01)(24.20,3.03)(24.30,3.04)(24.40,3.06)(24.50,3.07)(24.60,3.09)(24.70,3.10)(24.80,3.12)(24.90,3.13)(25.00,3.15)(25.10,3.17)(25.20,3.18)(25.30,3.19)(25.40,3.21)(25.50,3.23)(25.60,3.24)(25.70,3.25)(25.80,3.27)(25.90,3.29)(26.00,3.30)(26.10,3.31)(26.20,3.33)(26.30,3.34)(26.40,3.36)(26.50,3.38)(26.60,3.39)(26.70,3.40)(26.80,3.42)(26.90,3.44)(27.00,3.45)(27.10,3.46)(27.20,3.48)(27.30,3.49)(27.40,3.51)(27.50,3.52)(27.60,3.54)(27.70,3.55)(27.80,3.57)(27.90,3.59)(28.00,3.60)(28.10,3.61)(28.20,3.63)(28.30,3.64)(28.40,3.66)(28.50,3.67)(28.60,3.69)(28.70,3.70)(28.80,3.72)(28.90,3.73)(29.00,3.75)(29.10,3.77)(29.20,3.78)(29.30,3.79)(29.40,3.81)(29.50,3.83)(29.60,3.84)(29.70,3.85)(29.80,3.87)(29.90,3.89)(30.00,3.90)(30.10,3.92)(30.20,3.93)(30.30,3.94)(30.40,3.96)(30.50,3.98)(30.60,3.99)(30.70,4.00)(30.80,4.02)(30.90,4.04)(31.00,4.05)(31.10,4.06)(31.20,4.08)(31.30,4.09)(31.40,4.11)(31.50,4.12)(31.60,4.14)(31.70,4.15)(31.80,4.17)(31.90,4.19)(32.00,4.20)(32.10,4.21)(32.20,4.23)(32.30,4.24)(32.40,4.26)(32.50,4.27)(32.60,4.29)(32.70,4.30)(32.80,4.32)(32.90,4.34)(33.00,4.35)(33.10,4.36)(33.20,4.38)(33.30,4.39)(33.40,4.41)(33.50,4.42)(33.60,4.44)(33.70,4.46)(33.80,4.47)(33.90,4.48)(34.00,4.50)
        \psline[showpoints=false,linestyle=dashed,dash=2pt 2pt](4.00,0.00)(4.10,0.79)(4.20,0.91)(4.30,0.99)(4.40,1.07)(4.50,1.13)(4.60,1.19)(4.70,1.24)(4.80,1.29)(4.90,1.34)(5.00,1.39)(5.10,1.43)(5.20,1.47)(5.30,1.51)(5.40,1.55)(5.50,1.59)(5.60,1.63)(5.70,1.67)(5.80,1.71)(5.90,1.74)(6.00,1.78)(6.10,1.81)(6.20,1.85)(6.30,1.88)(6.40,1.92)(6.50,1.95)(6.60,1.99)(6.70,2.02)(6.80,2.05)(6.90,2.09)(7.00,2.12)(7.10,2.15)(7.20,2.19)(7.30,2.22)(7.40,2.25)(7.50,2.28)(7.60,2.32)(7.70,2.35)(7.80,2.38)(7.90,2.41)(8.00,2.45)(8.10,2.48)(8.20,2.51)(8.30,2.54)(8.40,2.58)(8.50,2.61)(8.60,2.64)(8.70,2.67)(8.80,2.71)(8.90,2.74)(9.00,2.77)(9.10,2.80)(9.20,2.84)(9.30,2.87)(9.40,2.90)(9.50,2.94)(9.60,2.97)(9.70,3.00)(9.80,3.04)(9.90,3.07)(10.00,3.10)(10.10,3.14)(10.20,3.17)(10.30,3.20)(10.40,3.24)(10.50,3.27)(10.60,3.31)(10.70,3.34)(10.80,3.37)(10.90,3.41)(11.00,3.44)(11.10,3.48)(11.20,3.51)(11.30,3.55)(11.40,3.58)(11.50,3.62)(11.60,3.65)(11.70,3.69)(11.80,3.73)(11.90,3.76)(12.00,3.80)(12.10,3.84)(12.20,3.87)(12.30,3.91)(12.40,3.95)(12.50,3.98)(12.60,4.02)(12.70,4.06)(12.80,4.10)(12.90,4.13)(13.00,4.17)(13.10,4.21)(13.20,4.25)(13.30,4.29)(13.40,4.33)(13.50,4.37)(13.60,4.41)(13.70,4.45)(13.80,4.49)(13.90,4.53)(14.00,4.57)(14.10,4.61)(14.20,4.65)(14.30,4.69)(14.40,4.73)(14.50,4.77)(14.60,4.81)(14.70,4.86)(14.80,4.90)(14.90,4.94)(15.00,4.99)(15.10,5.03)(15.20,5.07)(15.30,5.12)(15.40,5.16)(15.50,5.21)(15.60,5.25)(15.70,5.30)(15.80,5.34)(15.90,5.39)(16.00,5.43)(16.10,5.48)(16.20,5.53)(16.30,5.57)(16.40,5.62)(16.50,5.67)(16.60,5.72)(16.70,5.76)(16.80,5.81)(16.90,5.86)(17.00,5.91)(17.10,5.96)(17.20,6.01)(17.30,6.06)(17.40,6.11)(17.50,6.16)(17.60,6.22)(17.70,6.27)(17.80,6.32)(17.90,6.37)(18.00,6.43)(18.10,6.48)(18.20,6.53)(18.30,6.59)(18.40,6.64)(18.50,6.70)(18.60,6.75)(18.70,6.81)(18.80,6.87)(18.90,6.92)(19.00,6.98)(19.10,7.04)(19.20,7.10)(19.30,7.16)(19.40,7.22)(19.50,7.28)(19.60,7.34)(19.70,7.40)(19.80,7.46)(19.90,7.52)(20.00,7.58)(20.10,7.65)(20.20,7.71)(20.30,7.77)(20.40,7.84)(20.50,7.90)(20.60,7.97)(20.70,8.04)(20.80,8.10)(20.90,8.17)(21.00,8.24)(21.10,8.31)(21.20,8.38)(21.30,8.45)(21.40,8.52)(21.50,8.59)(21.60,8.66)(21.70,8.73)(21.80,8.80)(21.90,8.88)(22.00,8.95)(22.10,9.03)(22.20,9.10)(22.30,9.18)(22.40,9.25)(22.50,9.33)(22.60,9.41)(22.70,9.49)(22.80,9.57)(22.90,9.65)(23.00,9.73)(23.10,9.81)(23.20,9.90)(23.30,9.98)(23.40,10.06)(23.50,10.15)(23.60,10.24)(23.70,10.32)(23.80,10.41)(23.90,10.50)(24.00,10.59)(24.10,10.68)(24.20,10.77)(24.30,10.86)(24.40,10.95)(24.50,11.05)(24.60,11.14)(24.70,11.24)(24.80,11.34)(24.90,11.43)(25.00,11.53)(25.10,11.63)(25.20,11.73)(25.30,11.84)(25.40,11.94)(25.50,12.04)(25.60,12.15)(25.70,12.25)(25.80,12.36)(25.90,12.47)(26.00,12.58)(26.10,12.69)(26.20,12.80)(26.30,12.91)(26.40,13.03)(26.50,13.14)(26.60,13.26)(26.70,13.38)(26.80,13.50)(26.90,13.62)(27.00,13.74)(27.10,13.86)(27.20,13.99)(27.30,14.11)(27.40,14.24)(27.50,14.37)(27.60,14.50)(27.70,14.63)(27.80,14.76)(27.90,14.90)(28.00,15.03)(28.10,15.17)(28.20,15.31)(28.30,15.45)(28.40,15.59)(28.50,15.74)(28.60,15.88)(28.70,16.03)(28.80,16.18)(28.90,16.33)(29.00,16.49)(29.10,16.64)(29.20,16.80)(29.30,16.96)(29.40,17.12)(29.50,17.28)(29.60,17.44)(29.70,17.61)(29.80,17.78)(29.90,17.95)(30.00,18.12)(30.10,18.30)(30.20,18.47)(30.30,18.65)(30.40,18.83)(30.50,19.02)(30.60,19.20)(30.70,19.39)(30.80,19.58)(30.90,19.78)(31.00,19.97)(31.10,20.17)(31.20,20.37)(31.30,20.57)(31.40,20.78)(31.50,20.99)(31.60,21.20)(31.70,21.42)(31.80,21.63)(31.90,21.85)(32.00,22.08)(32.10,22.30)(32.20,22.53)(32.30,22.77)(32.40,23.00)(32.50,23.24)(32.60,23.49)(32.70,23.73)(32.80,23.98)(32.90,24.23)(33.00,24.49)(33.10,24.75)(33.20,25.02)(33.30,25.28)(33.40,25.56)(33.50,25.83)(33.60,26.11)(33.70,26.40)(33.80,26.68)(33.90,26.98)(34.00,27.27)
        \thl
        \psline{-}(58,-0.5)(58,0.5)\psline{-}(68,-0.5)(68,0.5)\psline{-}(78,-0.5)(78,0.5)
        \rput[tc](58,-3){$0.1$}\rput[tc](68,-3){$0.2$}\rput[tc](78,-3){$0.3$}
        \psline{-}(47.5,15)(48.5,15)\psline{-}(47.5,30)(48.5,30)
        \rput[mr](46.5,15){$\frac{1}{2}$}\rput[mr](46.5,30){$1$}
        \stl
        \psline[showpoints=false,linestyle=solid](48.00,0.00)(48.10,0.03)(48.20,0.06)(48.30,0.09)(48.40,0.12)(48.50,0.15)(48.60,0.18)(48.70,0.21)(48.80,0.24)(48.90,0.27)(49.00,0.30)(49.10,0.33)(49.20,0.36)(49.30,0.39)(49.40,0.42)(49.50,0.45)(49.60,0.48)(49.70,0.51)(49.80,0.54)(49.90,0.57)(50.00,0.60)(50.10,0.63)(50.20,0.66)(50.30,0.69)(50.40,0.72)(50.50,0.75)(50.60,0.78)(50.70,0.81)(50.80,0.84)(50.90,0.87)(51.00,0.90)(51.10,0.93)(51.20,0.96)(51.30,0.99)(51.40,1.02)(51.50,1.05)(51.60,1.08)(51.70,1.11)(51.80,1.14)(51.90,1.17)(52.00,1.20)(52.10,1.23)(52.20,1.26)(52.30,1.29)(52.40,1.32)(52.50,1.35)(52.60,1.38)(52.70,1.41)(52.80,1.44)(52.90,1.47)(53.00,1.50)(53.10,1.53)(53.20,1.56)(53.30,1.59)(53.40,1.62)(53.50,1.65)(53.60,1.68)(53.70,1.71)(53.80,1.74)(53.90,1.77)(54.00,1.80)(54.10,1.83)(54.20,1.86)(54.30,1.89)(54.40,1.92)(54.50,1.95)(54.60,1.98)(54.70,2.01)(54.80,2.04)(54.90,2.07)(55.00,2.10)(55.10,2.13)(55.20,2.16)(55.30,2.19)(55.40,2.22)(55.50,2.25)(55.60,2.28)(55.70,2.31)(55.80,2.34)(55.90,2.37)(56.00,2.40)(56.10,2.43)(56.20,2.46)(56.30,2.49)(56.40,2.52)(56.50,2.55)(56.60,2.58)(56.70,2.61)(56.80,2.64)(56.90,2.67)(57.00,2.70)(57.10,2.73)(57.20,2.76)(57.30,2.79)(57.40,2.82)(57.50,2.85)(57.60,2.88)(57.70,2.91)(57.80,2.94)(57.90,2.97)(58.00,3.00)(58.10,3.03)(58.20,3.06)(58.30,3.09)(58.40,3.12)(58.50,3.15)(58.60,3.18)(58.70,3.21)(58.80,3.24)(58.90,3.27)(59.00,3.30)(59.10,3.33)(59.20,3.36)(59.30,3.39)(59.40,3.42)(59.50,3.45)(59.60,3.48)(59.70,3.51)(59.80,3.54)(59.90,3.57)(60.00,3.60)(60.10,3.63)(60.20,3.66)(60.30,3.69)(60.40,3.72)(60.50,3.75)(60.60,3.78)(60.70,3.81)(60.80,3.84)(60.90,3.87)(61.00,3.90)(61.10,3.93)(61.20,3.96)(61.30,3.99)(61.40,4.02)(61.50,4.05)(61.60,4.08)(61.70,4.11)(61.80,4.14)(61.90,4.17)(62.00,4.20)(62.10,4.23)(62.20,4.26)(62.30,4.29)(62.40,4.32)(62.50,4.35)(62.60,4.38)(62.70,4.41)(62.80,4.44)(62.90,4.47)(63.00,4.50)(63.10,4.53)(63.20,4.56)(63.30,4.59)(63.40,4.62)(63.50,4.65)(63.60,4.68)(63.70,4.71)(63.80,4.74)(63.90,4.77)(64.00,4.80)(64.10,4.83)(64.20,4.86)(64.30,4.89)(64.40,4.92)(64.50,4.95)(64.60,4.98)(64.70,5.01)(64.80,5.04)(64.90,5.07)(65.00,5.10)(65.10,5.13)(65.20,5.16)(65.30,5.19)(65.40,5.22)(65.50,5.25)(65.60,5.28)(65.70,5.31)(65.80,5.34)(65.90,5.37)(66.00,5.40)(66.10,5.43)(66.20,5.46)(66.30,5.49)(66.40,5.52)(66.50,5.55)(66.60,5.58)(66.70,5.61)(66.80,5.64)(66.90,5.67)(67.00,5.70)(67.10,5.73)(67.20,5.76)(67.30,5.79)(67.40,5.82)(67.50,5.85)(67.60,5.88)(67.70,5.91)(67.80,5.94)(67.90,5.97)(68.00,6.00)(68.10,6.03)(68.20,6.06)(68.30,6.09)(68.40,6.12)(68.50,6.15)(68.60,6.18)(68.70,6.21)(68.80,6.24)(68.90,6.27)(69.00,6.30)(69.10,6.33)(69.20,6.36)(69.30,6.39)(69.40,6.42)(69.50,6.45)(69.60,6.48)(69.70,6.51)(69.80,6.54)(69.90,6.57)(70.00,6.60)(70.10,6.63)(70.20,6.66)(70.30,6.69)(70.40,6.72)(70.50,6.75)(70.60,6.78)(70.70,6.81)(70.80,6.84)(70.90,6.87)(71.00,6.90)(71.10,6.93)(71.20,6.96)(71.30,6.99)(71.40,7.02)(71.50,7.05)(71.60,7.08)(71.70,7.11)(71.80,7.14)(71.90,7.17)(72.00,7.20)(72.10,7.23)(72.20,7.26)(72.30,7.29)(72.40,7.32)(72.50,7.35)(72.60,7.38)(72.70,7.41)(72.80,7.44)(72.90,7.47)(73.00,7.50)(73.10,7.53)(73.20,7.56)(73.30,7.59)(73.40,7.62)(73.50,7.65)(73.60,7.68)(73.70,7.71)(73.80,7.74)(73.90,7.77)(74.00,7.80)(74.10,7.83)(74.20,7.86)(74.30,7.89)(74.40,7.92)(74.50,7.95)(74.60,7.98)(74.70,8.01)(74.80,8.04)(74.90,8.07)(75.00,8.10)(75.10,8.13)(75.20,8.16)(75.30,8.19)(75.40,8.22)(75.50,8.25)(75.60,8.28)(75.70,8.31)(75.80,8.34)(75.90,8.37)(76.00,8.40)(76.10,8.43)(76.20,8.46)(76.30,8.49)(76.40,8.52)(76.50,8.55)(76.60,8.58)(76.70,8.61)(76.80,8.64)(76.90,8.67)(77.00,8.70)(77.10,8.73)(77.20,8.76)(77.30,8.79)(77.40,8.82)(77.50,8.85)(77.60,8.88)(77.70,8.91)(77.80,8.94)(77.90,8.97)(78.00,9.00)
        \psline[showpoints=false,linestyle=dashed,dash=2pt 2pt](48.00,0.00)(48.10,2.77)(48.20,3.14)(48.30,3.41)(48.40,3.62)(48.50,3.81)(48.60,3.97)(48.70,4.12)(48.80,4.27)(48.90,4.40)(49.00,4.52)(49.10,4.64)(49.20,4.75)(49.30,4.86)(49.40,4.97)(49.50,5.07)(49.60,5.17)(49.70,5.27)(49.80,5.36)(49.90,5.45)(50.00,5.54)(50.10,5.63)(50.20,5.72)(50.30,5.80)(50.40,5.89)(50.50,5.97)(50.60,6.05)(50.70,6.13)(50.80,6.21)(50.90,6.29)(51.00,6.37)(51.10,6.45)(51.20,6.52)(51.30,6.60)(51.40,6.67)(51.50,6.75)(51.60,6.82)(51.70,6.90)(51.80,6.97)(51.90,7.04)(52.00,7.11)(52.10,7.18)(52.20,7.25)(52.30,7.33)(52.40,7.40)(52.50,7.47)(52.60,7.53)(52.70,7.60)(52.80,7.67)(52.90,7.74)(53.00,7.81)(53.10,7.88)(53.20,7.95)(53.30,8.01)(53.40,8.08)(53.50,8.15)(53.60,8.21)(53.70,8.28)(53.80,8.35)(53.90,8.41)(54.00,8.48)(54.10,8.55)(54.20,8.61)(54.30,8.68)(54.40,8.74)(54.50,8.81)(54.60,8.88)(54.70,8.94)(54.80,9.01)(54.90,9.07)(55.00,9.14)(55.10,9.20)(55.20,9.27)(55.30,9.33)(55.40,9.40)(55.50,9.46)(55.60,9.53)(55.70,9.59)(55.80,9.66)(55.90,9.72)(56.00,9.79)(56.10,9.85)(56.20,9.92)(56.30,9.98)(56.40,10.05)(56.50,10.11)(56.60,10.18)(56.70,10.24)(56.80,10.31)(56.90,10.37)(57.00,10.44)(57.10,10.50)(57.20,10.57)(57.30,10.63)(57.40,10.70)(57.50,10.76)(57.60,10.83)(57.70,10.89)(57.80,10.96)(57.90,11.02)(58.00,11.09)(58.10,11.15)(58.20,11.22)(58.30,11.28)(58.40,11.35)(58.50,11.42)(58.60,11.48)(58.70,11.55)(58.80,11.61)(58.90,11.68)(59.00,11.75)(59.10,11.81)(59.20,11.88)(59.30,11.94)(59.40,12.01)(59.50,12.08)(59.60,12.14)(59.70,12.21)(59.80,12.28)(59.90,12.34)(60.00,12.41)(60.10,12.48)(60.20,12.54)(60.30,12.61)(60.40,12.68)(60.50,12.75)(60.60,12.81)(60.70,12.88)(60.80,12.95)(60.90,13.02)(61.00,13.09)(61.10,13.15)(61.20,13.22)(61.30,13.29)(61.40,13.36)(61.50,13.43)(61.60,13.50)(61.70,13.57)(61.80,13.64)(61.90,13.70)(62.00,13.77)(62.10,13.84)(62.20,13.91)(62.30,13.98)(62.40,14.05)(62.50,14.12)(62.60,14.19)(62.70,14.26)(62.80,14.34)(62.90,14.41)(63.00,14.48)(63.10,14.55)(63.20,14.62)(63.30,14.69)(63.40,14.76)(63.50,14.83)(63.60,14.91)(63.70,14.98)(63.80,15.05)(63.90,15.12)(64.00,15.20)(64.10,15.27)(64.20,15.34)(64.30,15.41)(64.40,15.49)(64.50,15.56)(64.60,15.64)(64.70,15.71)(64.80,15.78)(64.90,15.86)(65.00,15.93)(65.10,16.01)(65.20,16.08)(65.30,16.16)(65.40,16.23)(65.50,16.31)(65.60,16.38)(65.70,16.46)(65.80,16.54)(65.90,16.61)(66.00,16.69)(66.10,16.77)(66.20,16.84)(66.30,16.92)(66.40,17.00)(66.50,17.07)(66.60,17.15)(66.70,17.23)(66.80,17.31)(66.90,17.39)(67.00,17.47)(67.10,17.55)(67.20,17.62)(67.30,17.70)(67.40,17.78)(67.50,17.86)(67.60,17.94)(67.70,18.02)(67.80,18.10)(67.90,18.19)(68.00,18.27)(68.10,18.35)(68.20,18.43)(68.30,18.51)(68.40,18.59)(68.50,18.68)(68.60,18.76)(68.70,18.84)(68.80,18.92)(68.90,19.01)(69.00,19.09)(69.10,19.17)(69.20,19.26)(69.30,19.34)(69.40,19.43)(69.50,19.51)(69.60,19.60)(69.70,19.68)(69.80,19.77)(69.90,19.86)(70.00,19.94)(70.10,20.03)(70.20,20.12)(70.30,20.20)(70.40,20.29)(70.50,20.38)(70.60,20.47)(70.70,20.55)(70.80,20.64)(70.90,20.73)(71.00,20.82)(71.10,20.91)(71.20,21.00)(71.30,21.09)(71.40,21.18)(71.50,21.27)(71.60,21.36)(71.70,21.45)(71.80,21.55)(71.90,21.64)(72.00,21.73)(72.10,21.82)(72.20,21.92)(72.30,22.01)(72.40,22.10)(72.50,22.20)(72.60,22.29)(72.70,22.38)(72.80,22.48)(72.90,22.57)(73.00,22.67)(73.10,22.77)(73.20,22.86)(73.30,22.96)(73.40,23.06)(73.50,23.15)(73.60,23.25)(73.70,23.35)(73.80,23.45)(73.90,23.55)(74.00,23.65)(74.10,23.74)(74.20,23.84)(74.30,23.94)(74.40,24.05)(74.50,24.15)(74.60,24.25)(74.70,24.35)(74.80,24.45)(74.90,24.55)(75.00,24.66)(75.10,24.76)(75.20,24.86)(75.30,24.97)(75.40,25.07)(75.50,25.18)(75.60,25.28)(75.70,25.39)(75.80,25.49)(75.90,25.60)(76.00,25.70)(76.10,25.81)(76.20,25.92)(76.30,26.03)(76.40,26.14)(76.50,26.24)(76.60,26.35)(76.70,26.46)(76.80,26.57)(76.90,26.68)(77.00,26.79)(77.10,26.91)(77.20,27.02)(77.30,27.13)(77.40,27.24)(77.50,27.36)(77.60,27.47)(77.70,27.58)(77.80,27.70)(77.90,27.81)(78.00,27.93)
    \end{pspicture}
    \vspace*{-0.2cm}
    \end{center}
    \caption{Comparison of SCC bounds for WSS (solid line) and NRM (dashed line) disturbances,
    according to Table \ref{Tab:SCC}.\label{Fig:SCCBounds}}
\end{figure}
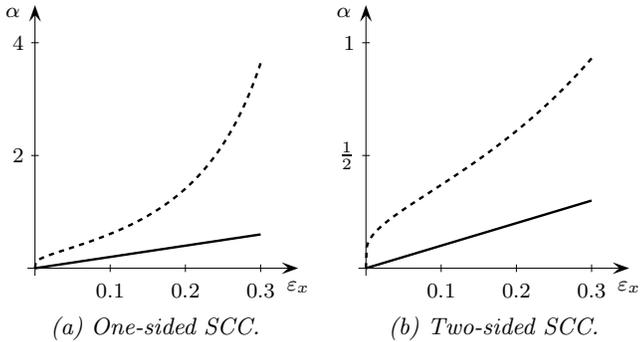

\section{Joint Chance Constraints (JCCs)}\label{Sec:JCC}

Joint Chance Constraints (JCCs) are of the form \eqref{Equ:ChanceConstr}, where the constraint set $\BX$ or $\BU$ is an ellipsoid that is symmetric with respect to the origin:
\begin{align}
    \Pb\bigl[x\tp G^{-1}x\leq d\bigr]\geq 1-\epx\ec\label{Equ:StateJCC}\\
    \Pb\bigl[u\tp F^{-1}u\leq c\bigr]\geq 1-\epu\ef\label{Equ:InputJCC}
\end{align}
Here $G\in\BR^{n\times n}$ and $F\in\BR^{m\times m}$ are positive definite matrices and $d,c\in\BR_{+}$ are positive scalars.

In the remainder of this section, the reformulation of \eqref{Equ:StateJCC} and \eqref{Equ:InputJCC} into LMIs is described.

\begin{remark}[Polytopic Constraints]
	JCCs can also be used to accommodate polytopic sets $\BX$ or $\BU$ in a conservative
	manner, through an inner approximation by a maximum inscribed ellipsoid \cite[Thm.\,1]
	{ZhouCog:2013}. This ellipsoid can be computed efficiently by a convex SDP, as described by
	\cite[Sec.\,8.4.2]{BoydVan:1996}.
\end{remark}

In contrast to an SCC, a JCC requires the `tail bound' of a multivariate distribution, rather than a univariate one. This problem has previously emerged in the theory of robust control. Several approaches have been proposed in the literature, \eg \cite[Sec.\,7.4.1]{BoydVan:1996}. 

Here the \emph{generalized Chebyshev inequality} of \cite{ChenZhou:1997} shall be invoked. To this end, let $\csd_n:\BR_{0+}\to[0,1]$ denote the chi-squared cumulative distribution function with parameter $n$ \cite[Equ.\,26.4.1]{Abramowitz:1970}.

\begin{lemma}[JCC on the State]\label{The:JCCState}
	Let $X$ be any matrix such that $X\succeq \bX$. 
	Then the JCC on the state \eqref{Equ:StateJCC} is satisfied in the case of a WSS disturbance if
	\begin{equation}\label{Equ:JCCWSSState}
		X\preceq\frac{d\epx}{n}G\ec
	\end{equation}
	and in the case of a NRM disturbance if
	\begin{equation}\label{Equ:JCCNRMState}
		X\preceq \bigl(d/\csd_n^{-1}(1-\epx)\bigr)G\ef
	\end{equation}
\end{lemma}

\begin{proof}
	First, consider the case of a WSS disturbance. Since $x$ is distributed with mean $0$ and
	covariance $\bX$, the generalized Chebyshev inequality \cite[Thm.\,2.1]{ChenZhou:1997} gives
	\begin{align}\label{Equ:TheJCCWSSEqu1}
		&\Pb\bigl[x\tp \bX^{-1}x>n/\epx\bigr]\leq\epx\ec\nonumber\\
		\text{or}\qquad&\Pb\bigl[(d\epx/n)x\tp \bX^{-1}x>d\bigr]\leq\epx\ef\hspace*{1.5cm}
	\end{align}
	Now suppose that \eqref{Equ:JCCWSSState} is satisfied, thus
	\begin{equation*}
		\bX \preceq X\preceq (d\epx/n)G\quad\Longrightarrow\quad (d\epx/n)\bX^{-1}\succeq G^{-1}\ef
	\end{equation*}
	Then the following set inclusion holds
	\begin{equation*}
		\bigl\{x\,:\,x\tp G^{-1}x>d\bigr\}\subseteq\bigl\{x\,:\,(d\epx/n)x\tp \bX^{-1}x>d\bigr\}\ec
	\end{equation*}
	and therefore
	\begin{equation*}
		\Pb\bigl[x\tp G^{-1}x>d\bigr]\leq\Pb\bigl[(d\epx/n)x\tp \bX^{-1}x>d\bigr]\leq\epx\ec
	\end{equation*}
	so the state constraint \eqref{Equ:StateJCC} is satisfied.
	
	Second, consider the case of a NRM disturbance. Define the auxiliary random variable 
	$z:=\bX^{-1/2}x$. As $z\sim\CN(0,I)$, the random variable
    \begin{equation*}
        x\tp \bX^{-1}x=z\tp z\sim\chi^2(n)\ec
    \end{equation*}
    where $\chi^2(n)$ denotes the chi-squared distribution with parameter $n$ \cite[Sec.\,26.4]
    {Abramowitz:1970}. With its cumulative distribution function,
    \begin{equation*}
        \Pb\bigl[z\tp z\leq\csd_n^{-1}(1-\epx)\bigr]=1-\epx
    \end{equation*}
    and therefore
    \begin{equation}\label{Equ:TheJCCNRMEqu1}
        \Pb\bigl[\bigl(d/\csd_n^{-1}(1-\epx)\bigr)x\tp \bX^{-1}x\leq d\bigr]=1-\epx\ef
    \end{equation}
    Now suppose that \eqref{Equ:JCCNRMState} is satisfied, thus
    \begin{multline*}
		\bX \preceq X\preceq \bigl(d/\csd_n^{-1}(1-\epx)\bigr)G\\
		\Longrightarrow\quad \bigl(d/\csd_n^{-1}(1-\epx)\bigr)\bX^{-1}\succeq G^{-1}\ef
	\end{multline*}
    Then the following set inclusion holds
    \begin{equation*}
    \begin{split}
		\bigl\{x\,:\,x\tp G^{-1}x\leq d\bigr\}\supseteq \hfill\\
		&\hspace{-10ex}\bigl\{x\,:\,\bigl(d/\csd_n^{-1}(1-\epx)\bigr)x\tp \bX^{-1}x\leq d\bigr\}
		\end{split}
	\end{equation*}
	and therefore
	\begin{multline*}
		\Pb\bigl[x\tp G^{-1}x\leq d\bigr]\geq\\
		\Pb\bigl[\bigl(d/\csd_n^{-1}(1-\epx)\bigr)x\tp \bX^{-1}x\leq d\bigr]=
		1-\epx\ec
	\end{multline*}
    so the state constraint \eqref{Equ:StateJCC} is satisfied.
\end{proof}

\begin{lemma}[JCC on the Input]\label{The:JCCInput}
	Let $X$ be any matrix such that $X\succeq \bX$. 
	The JCC on the input \eqref{Equ:InputJCC} is satisfied in the case a WSS disturbance if
	\begin{equation}\label{Equ:JCCWSSInput}
		\begin{bmatrix} (c\epu/m)\cdot F & Y \\ Y\tp & X \end{bmatrix} \succeq 0\ec
	\end{equation}
	and in the case of a NRM disturbance if
	\begin{equation}\label{Equ:JCCNRMInput}
		\begin{bmatrix} c/\csd_m^{-1}(1-\epu)\cdot F & Y \\ Y\tp & X \end{bmatrix} \succeq 0\ef
	\end{equation}
\end{lemma}

\begin{proof}
	Since $\bX \preceq X$ implies $K\bX K\tp \preceq K X K\tp$, the same argument as in the proof of
	Lemma \ref{The:JCCState} can be employed to show that 
	\begin{equation}
		K X K\tp\preceq\frac{d\epu}{n}F
	\end{equation}
	and 
	\begin{equation}
		K X K\tp\preceq \bigl(d/\csd_n^{-1}(1-\epu)\bigr)F
	\end{equation}
	are sufficient conditions to satisfy \eqref{Equ:InputJCC} in the cases of a WSS and NRM
	disturbance, respectively. Note that the implicit requirement that $(KXK\tp)^{-1}$ is invertible
	amounts to the assumption that $K$ has full row rank.
	
	The lemma then follows by another application of the Schur complement \cite[Sec.\,2.1]
	{BoydEtAl:1994}. 
\end{proof}

For a comparison of the bounds in Lemma \ref{The:JCCState} (analogously, Lemma \ref{The:JCCInput}), consider the factor
\begin{equation}\label{Equ:CompJCCBound}
	\gamma_{n}=\frac{\epx}{n}\quad\text{and}\quad\gamma_{n}=1/\csd_n^{-1}(1-\epx)
\end{equation}
for WSS and NRM disturbances, respectively. This factor multiplies $d\cdot G$ in \eqref{Equ:JCCWSSState} and \eqref{Equ:JCCNRMState}. Note that it depends on the dimension $n$ of the JCC. Figure \ref{Fig:JCCBounds} compares these factors for exemplary dimensions of $n=2$ and $n=3$, showing again that the NRM assumption leads to less restrictive bounds on $X$.

\begin{figure}[H]
    \begin{center}
	\begin{pspicture}(2,-10)(85,40)
		\footnotesize
	    \rput[tc](20,-8){\textit{(a) Dimension $n=2$.}}
	    \rput[tc](65,-8){\textit{(b) Dimension $n=3$.}}
	    \scriptsize
        \thl
        \psline[arrowsize=4.5pt]{->}(3,0)(39,0)
        \psline[arrowsize=4.5pt]{->}(4,-1)(4,35)
        \psline[arrowsize=4.5pt]{->}(47,0)(83,0)
        \psline[arrowsize=4.5pt]{->}(48,-1)(48,35)
        \rput[tc](39,-2.5){$\epx$}
        \rput[mr](2,34){$\gamma_2$}
        \rput[tc](83,-2.5){$\epx$}
        \rput[mr](46,34){$\gamma_3$}
        \psline{-}(14,-0.5)(14,0.5)\psline{-}(24,-0.5)(24,0.5)\psline{-}(34,-0.5)(34,0.5)
        \rput[tc](14,-3){$0.1$}\rput[tc](24,-3){$0.2$}\rput[tc](34,-3){$0.3$}
        \psline{-}(3.5,15)(4.5,15)\psline{-}(3.5,30)(4.5,30)
        \rput[mr](2.5,15){$\frac{1}{4}$}\rput[mr](2.5,30){$\frac{1}{2}$}
        \stl
        \psline[showpoints=false,linestyle=solid](4.00,0.00)(4.10,0.03)(4.20,0.06)(4.30,0.09)(4.40,0.12)(4.50,0.15)(4.60,0.18)(4.70,0.21)(4.80,0.24)(4.90,0.27)(5.00,0.30)(5.10,0.33)(5.20,0.36)(5.30,0.39)(5.40,0.42)(5.50,0.45)(5.60,0.48)(5.70,0.51)(5.80,0.54)(5.90,0.57)(6.00,0.60)(6.10,0.63)(6.20,0.66)(6.30,0.69)(6.40,0.72)(6.50,0.75)(6.60,0.78)(6.70,0.81)(6.80,0.84)(6.90,0.87)(7.00,0.90)(7.10,0.93)(7.20,0.96)(7.30,0.99)(7.40,1.02)(7.50,1.05)(7.60,1.08)(7.70,1.11)(7.80,1.14)(7.90,1.17)(8.00,1.20)(8.10,1.23)(8.20,1.26)(8.30,1.29)(8.40,1.32)(8.50,1.35)(8.60,1.38)(8.70,1.41)(8.80,1.44)(8.90,1.47)(9.00,1.50)(9.10,1.53)(9.20,1.56)(9.30,1.59)(9.40,1.62)(9.50,1.65)(9.60,1.68)(9.70,1.71)(9.80,1.74)(9.90,1.77)(10.00,1.80)(10.10,1.83)(10.20,1.86)(10.30,1.89)(10.40,1.92)(10.50,1.95)(10.60,1.98)(10.70,2.01)(10.80,2.04)(10.90,2.07)(11.00,2.10)(11.10,2.13)(11.20,2.16)(11.30,2.19)(11.40,2.22)(11.50,2.25)(11.60,2.28)(11.70,2.31)(11.80,2.34)(11.90,2.37)(12.00,2.40)(12.10,2.43)(12.20,2.46)(12.30,2.49)(12.40,2.52)(12.50,2.55)(12.60,2.58)(12.70,2.61)(12.80,2.64)(12.90,2.67)(13.00,2.70)(13.10,2.73)(13.20,2.76)(13.30,2.79)(13.40,2.82)(13.50,2.85)(13.60,2.88)(13.70,2.91)(13.80,2.94)(13.90,2.97)(14.00,3.00)(14.10,3.03)(14.20,3.06)(14.30,3.09)(14.40,3.12)(14.50,3.15)(14.60,3.18)(14.70,3.21)(14.80,3.24)(14.90,3.27)(15.00,3.30)(15.10,3.33)(15.20,3.36)(15.30,3.39)(15.40,3.42)(15.50,3.45)(15.60,3.48)(15.70,3.51)(15.80,3.54)(15.90,3.57)(16.00,3.60)(16.10,3.63)(16.20,3.66)(16.30,3.69)(16.40,3.72)(16.50,3.75)(16.60,3.78)(16.70,3.81)(16.80,3.84)(16.90,3.87)(17.00,3.90)(17.10,3.93)(17.20,3.96)(17.30,3.99)(17.40,4.02)(17.50,4.05)(17.60,4.08)(17.70,4.11)(17.80,4.14)(17.90,4.17)(18.00,4.20)(18.10,4.23)(18.20,4.26)(18.30,4.29)(18.40,4.32)(18.50,4.35)(18.60,4.38)(18.70,4.41)(18.80,4.44)(18.90,4.47)(19.00,4.50)(19.10,4.53)(19.20,4.56)(19.30,4.59)(19.40,4.62)(19.50,4.65)(19.60,4.68)(19.70,4.71)(19.80,4.74)(19.90,4.77)(20.00,4.80)(20.10,4.83)(20.20,4.86)(20.30,4.89)(20.40,4.92)(20.50,4.95)(20.60,4.98)(20.70,5.01)(20.80,5.04)(20.90,5.07)(21.00,5.10)(21.10,5.13)(21.20,5.16)(21.30,5.19)(21.40,5.22)(21.50,5.25)(21.60,5.28)(21.70,5.31)(21.80,5.34)(21.90,5.37)(22.00,5.40)(22.10,5.43)(22.20,5.46)(22.30,5.49)(22.40,5.52)(22.50,5.55)(22.60,5.58)(22.70,5.61)(22.80,5.64)(22.90,5.67)(23.00,5.70)(23.10,5.73)(23.20,5.76)(23.30,5.79)(23.40,5.82)(23.50,5.85)(23.60,5.88)(23.70,5.91)(23.80,5.94)(23.90,5.97)(24.00,6.00)(24.10,6.03)(24.20,6.06)(24.30,6.09)(24.40,6.12)(24.50,6.15)(24.60,6.18)(24.70,6.21)(24.80,6.24)(24.90,6.27)(25.00,6.30)(25.10,6.33)(25.20,6.36)(25.30,6.39)(25.40,6.42)(25.50,6.45)(25.60,6.48)(25.70,6.51)(25.80,6.54)(25.90,6.57)(26.00,6.60)(26.10,6.63)(26.20,6.66)(26.30,6.69)(26.40,6.72)(26.50,6.75)(26.60,6.78)(26.70,6.81)(26.80,6.84)(26.90,6.87)(27.00,6.90)(27.10,6.93)(27.20,6.96)(27.30,6.99)(27.40,7.02)(27.50,7.05)(27.60,7.08)(27.70,7.11)(27.80,7.14)(27.90,7.17)(28.00,7.20)(28.10,7.23)(28.20,7.26)(28.30,7.29)(28.40,7.32)(28.50,7.35)(28.60,7.38)(28.70,7.41)(28.80,7.44)(28.90,7.47)(29.00,7.50)(29.10,7.53)(29.20,7.56)(29.30,7.59)(29.40,7.62)(29.50,7.65)(29.60,7.68)(29.70,7.71)(29.80,7.74)(29.90,7.77)(30.00,7.80)(30.10,7.83)(30.20,7.86)(30.30,7.89)(30.40,7.92)(30.50,7.95)(30.60,7.98)(30.70,8.01)(30.80,8.04)(30.90,8.07)(31.00,8.10)(31.10,8.13)(31.20,8.16)(31.30,8.19)(31.40,8.22)(31.50,8.25)(31.60,8.28)(31.70,8.31)(31.80,8.34)(31.90,8.37)(32.00,8.40)(32.10,8.43)(32.20,8.46)(32.30,8.49)(32.40,8.52)(32.50,8.55)(32.60,8.58)(32.70,8.61)(32.80,8.64)(32.90,8.67)(33.00,8.70)(33.10,8.73)(33.20,8.76)(33.30,8.79)(33.40,8.82)(33.50,8.85)(33.60,8.88)(33.70,8.91)(33.80,8.94)(33.90,8.97)(34.00,9.00)
        \psline[showpoints=false,linestyle=dashed,dash=2pt 2pt](4.00,0.00)(4.10,4.34)(4.20,4.83)(4.30,5.16)(4.40,5.43)(4.50,5.66)(4.60,5.86)(4.70,6.05)(4.80,6.21)(4.90,6.37)(5.00,6.51)(5.10,6.65)(5.20,6.78)(5.30,6.91)(5.40,7.03)(5.50,7.14)(5.60,7.25)(5.70,7.36)(5.80,7.47)(5.90,7.57)(6.00,7.67)(6.10,7.77)(6.20,7.86)(6.30,7.95)(6.40,8.04)(6.50,8.13)(6.60,8.22)(6.70,8.31)(6.80,8.39)(6.90,8.47)(7.00,8.56)(7.10,8.64)(7.20,8.72)(7.30,8.79)(7.40,8.87)(7.50,8.95)(7.60,9.02)(7.70,9.10)(7.80,9.17)(7.90,9.25)(8.00,9.32)(8.10,9.39)(8.20,9.46)(8.30,9.53)(8.40,9.60)(8.50,9.67)(8.60,9.74)(8.70,9.81)(8.80,9.88)(8.90,9.95)(9.00,10.01)(9.10,10.08)(9.20,10.15)(9.30,10.21)(9.40,10.28)(9.50,10.34)(9.60,10.41)(9.70,10.47)(9.80,10.54)(9.90,10.60)(10.00,10.66)(10.10,10.73)(10.20,10.79)(10.30,10.85)(10.40,10.91)(10.50,10.98)(10.60,11.04)(10.70,11.10)(10.80,11.16)(10.90,11.22)(11.00,11.28)(11.10,11.34)(11.20,11.40)(11.30,11.46)(11.40,11.52)(11.50,11.58)(11.60,11.64)(11.70,11.70)(11.80,11.76)(11.90,11.82)(12.00,11.88)(12.10,11.94)(12.20,12.00)(12.30,12.05)(12.40,12.11)(12.50,12.17)(12.60,12.23)(12.70,12.29)(12.80,12.34)(12.90,12.40)(13.00,12.46)(13.10,12.52)(13.20,12.57)(13.30,12.63)(13.40,12.69)(13.50,12.74)(13.60,12.80)(13.70,12.86)(13.80,12.92)(13.90,12.97)(14.00,13.03)(14.10,13.09)(14.20,13.14)(14.30,13.20)(14.40,13.25)(14.50,13.31)(14.60,13.37)(14.70,13.42)(14.80,13.48)(14.90,13.54)(15.00,13.59)(15.10,13.65)(15.20,13.70)(15.30,13.76)(15.40,13.81)(15.50,13.87)(15.60,13.93)(15.70,13.98)(15.80,14.04)(15.90,14.09)(16.00,14.15)(16.10,14.20)(16.20,14.26)(16.30,14.32)(16.40,14.37)(16.50,14.43)(16.60,14.48)(16.70,14.54)(16.80,14.59)(16.90,14.65)(17.00,14.70)(17.10,14.76)(17.20,14.82)(17.30,14.87)(17.40,14.93)(17.50,14.98)(17.60,15.04)(17.70,15.09)(17.80,15.15)(17.90,15.20)(18.00,15.26)(18.10,15.31)(18.20,15.37)(18.30,15.42)(18.40,15.48)(18.50,15.54)(18.60,15.59)(18.70,15.65)(18.80,15.70)(18.90,15.76)(19.00,15.81)(19.10,15.87)(19.20,15.92)(19.30,15.98)(19.40,16.04)(19.50,16.09)(19.60,16.15)(19.70,16.20)(19.80,16.26)(19.90,16.31)(20.00,16.37)(20.10,16.43)(20.20,16.48)(20.30,16.54)(20.40,16.59)(20.50,16.65)(20.60,16.71)(20.70,16.76)(20.80,16.82)(20.90,16.87)(21.00,16.93)(21.10,16.99)(21.20,17.04)(21.30,17.10)(21.40,17.16)(21.50,17.21)(21.60,17.27)(21.70,17.32)(21.80,17.38)(21.90,17.44)(22.00,17.49)(22.10,17.55)(22.20,17.61)(22.30,17.67)(22.40,17.72)(22.50,17.78)(22.60,17.84)(22.70,17.89)(22.80,17.95)(22.90,18.01)(23.00,18.06)(23.10,18.12)(23.20,18.18)(23.30,18.24)(23.40,18.29)(23.50,18.35)(23.60,18.41)(23.70,18.47)(23.80,18.52)(23.90,18.58)(24.00,18.64)(24.10,18.70)(24.20,18.76)(24.30,18.81)(24.40,18.87)(24.50,18.93)(24.60,18.99)(24.70,19.05)(24.80,19.11)(24.90,19.16)(25.00,19.22)(25.10,19.28)(25.20,19.34)(25.30,19.40)(25.40,19.46)(25.50,19.52)(25.60,19.58)(25.70,19.64)(25.80,19.69)(25.90,19.75)(26.00,19.81)(26.10,19.87)(26.20,19.93)(26.30,19.99)(26.40,20.05)(26.50,20.11)(26.60,20.17)(26.70,20.23)(26.80,20.29)(26.90,20.35)(27.00,20.41)(27.10,20.47)(27.20,20.53)(27.30,20.59)(27.40,20.65)(27.50,20.72)(27.60,20.78)(27.70,20.84)(27.80,20.90)(27.90,20.96)(28.00,21.02)(28.10,21.08)(28.20,21.14)(28.30,21.21)(28.40,21.27)(28.50,21.33)(28.60,21.39)(28.70,21.45)(28.80,21.52)(28.90,21.58)(29.00,21.64)(29.10,21.70)(29.20,21.77)(29.30,21.83)(29.40,21.89)(29.50,21.95)(29.60,22.02)(29.70,22.08)(29.80,22.14)(29.90,22.21)(30.00,22.27)(30.10,22.33)(30.20,22.40)(30.30,22.46)(30.40,22.53)(30.50,22.59)(30.60,22.65)(30.70,22.72)(30.80,22.78)(30.90,22.85)(31.00,22.91)(31.10,22.98)(31.20,23.04)(31.30,23.11)(31.40,23.17)(31.50,23.24)(31.60,23.30)(31.70,23.37)(31.80,23.44)(31.90,23.50)(32.00,23.57)(32.10,23.63)(32.20,23.70)(32.30,23.77)(32.40,23.83)(32.50,23.90)(32.60,23.97)(32.70,24.03)(32.80,24.10)(32.90,24.17)(33.00,24.24)(33.10,24.30)(33.20,24.37)(33.30,24.44)(33.40,24.51)(33.50,24.57)(33.60,24.64)(33.70,24.71)(33.80,24.78)(33.90,24.85)(34.00,24.92)
        \thl
        \psline{-}(58,-0.5)(58,0.5)\psline{-}(68,-0.5)(68,0.5)\psline{-}(78,-0.5)(78,0.5)
        \rput[tc](58,-3){$0.1$}\rput[tc](68,-3){$0.2$}\rput[tc](78,-3){$0.3$}
        \psline{-}(47.5,15)(48.5,15)\psline{-}(47.5,30)(48.5,30)
        \rput[mr](46.5,15){$\frac{1}{4}$}\rput[mr](46.5,30){$\frac{1}{2}$}
        \stl
        \psline[showpoints=false,linestyle=solid](48.00,0.00)(48.10,0.02)(48.20,0.04)(48.30,0.06)(48.40,0.08)(48.50,0.10)(48.60,0.12)(48.70,0.14)(48.80,0.16)(48.90,0.18)(49.00,0.20)(49.10,0.22)(49.20,0.24)(49.30,0.26)(49.40,0.28)(49.50,0.30)(49.60,0.32)(49.70,0.34)(49.80,0.36)(49.90,0.38)(50.00,0.40)(50.10,0.42)(50.20,0.44)(50.30,0.46)(50.40,0.48)(50.50,0.50)(50.60,0.52)(50.70,0.54)(50.80,0.56)(50.90,0.58)(51.00,0.60)(51.10,0.62)(51.20,0.64)(51.30,0.66)(51.40,0.68)(51.50,0.70)(51.60,0.72)(51.70,0.74)(51.80,0.76)(51.90,0.78)(52.00,0.80)(52.10,0.82)(52.20,0.84)(52.30,0.86)(52.40,0.88)(52.50,0.90)(52.60,0.92)(52.70,0.94)(52.80,0.96)(52.90,0.98)(53.00,1.00)(53.10,1.02)(53.20,1.04)(53.30,1.06)(53.40,1.08)(53.50,1.10)(53.60,1.12)(53.70,1.14)(53.80,1.16)(53.90,1.18)(54.00,1.20)(54.10,1.22)(54.20,1.24)(54.30,1.26)(54.40,1.28)(54.50,1.30)(54.60,1.32)(54.70,1.34)(54.80,1.36)(54.90,1.38)(55.00,1.40)(55.10,1.42)(55.20,1.44)(55.30,1.46)(55.40,1.48)(55.50,1.50)(55.60,1.52)(55.70,1.54)(55.80,1.56)(55.90,1.58)(56.00,1.60)(56.10,1.62)(56.20,1.64)(56.30,1.66)(56.40,1.68)(56.50,1.70)(56.60,1.72)(56.70,1.74)(56.80,1.76)(56.90,1.78)(57.00,1.80)(57.10,1.82)(57.20,1.84)(57.30,1.86)(57.40,1.88)(57.50,1.90)(57.60,1.92)(57.70,1.94)(57.80,1.96)(57.90,1.98)(58.00,2.00)(58.10,2.02)(58.20,2.04)(58.30,2.06)(58.40,2.08)(58.50,2.10)(58.60,2.12)(58.70,2.14)(58.80,2.16)(58.90,2.18)(59.00,2.20)(59.10,2.22)(59.20,2.24)(59.30,2.26)(59.40,2.28)(59.50,2.30)(59.60,2.32)(59.70,2.34)(59.80,2.36)(59.90,2.38)(60.00,2.40)(60.10,2.42)(60.20,2.44)(60.30,2.46)(60.40,2.48)(60.50,2.50)(60.60,2.52)(60.70,2.54)(60.80,2.56)(60.90,2.58)(61.00,2.60)(61.10,2.62)(61.20,2.64)(61.30,2.66)(61.40,2.68)(61.50,2.70)(61.60,2.72)(61.70,2.74)(61.80,2.76)(61.90,2.78)(62.00,2.80)(62.10,2.82)(62.20,2.84)(62.30,2.86)(62.40,2.88)(62.50,2.90)(62.60,2.92)(62.70,2.94)(62.80,2.96)(62.90,2.98)(63.00,3.00)(63.10,3.02)(63.20,3.04)(63.30,3.06)(63.40,3.08)(63.50,3.10)(63.60,3.12)(63.70,3.14)(63.80,3.16)(63.90,3.18)(64.00,3.20)(64.10,3.22)(64.20,3.24)(64.30,3.26)(64.40,3.28)(64.50,3.30)(64.60,3.32)(64.70,3.34)(64.80,3.36)(64.90,3.38)(65.00,3.40)(65.10,3.42)(65.20,3.44)(65.30,3.46)(65.40,3.48)(65.50,3.50)(65.60,3.52)(65.70,3.54)(65.80,3.56)(65.90,3.58)(66.00,3.60)(66.10,3.62)(66.20,3.64)(66.30,3.66)(66.40,3.68)(66.50,3.70)(66.60,3.72)(66.70,3.74)(66.80,3.76)(66.90,3.78)(67.00,3.80)(67.10,3.82)(67.20,3.84)(67.30,3.86)(67.40,3.88)(67.50,3.90)(67.60,3.92)(67.70,3.94)(67.80,3.96)(67.90,3.98)(68.00,4.00)(68.10,4.02)(68.20,4.04)(68.30,4.06)(68.40,4.08)(68.50,4.10)(68.60,4.12)(68.70,4.14)(68.80,4.16)(68.90,4.18)(69.00,4.20)(69.10,4.22)(69.20,4.24)(69.30,4.26)(69.40,4.28)(69.50,4.30)(69.60,4.32)(69.70,4.34)(69.80,4.36)(69.90,4.38)(70.00,4.40)(70.10,4.42)(70.20,4.44)(70.30,4.46)(70.40,4.48)(70.50,4.50)(70.60,4.52)(70.70,4.54)(70.80,4.56)(70.90,4.58)(71.00,4.60)(71.10,4.62)(71.20,4.64)(71.30,4.66)(71.40,4.68)(71.50,4.70)(71.60,4.72)(71.70,4.74)(71.80,4.76)(71.90,4.78)(72.00,4.80)(72.10,4.82)(72.20,4.84)(72.30,4.86)(72.40,4.88)(72.50,4.90)(72.60,4.92)(72.70,4.94)(72.80,4.96)(72.90,4.98)(73.00,5.00)(73.10,5.02)(73.20,5.04)(73.30,5.06)(73.40,5.08)(73.50,5.10)(73.60,5.12)(73.70,5.14)(73.80,5.16)(73.90,5.18)(74.00,5.20)(74.10,5.22)(74.20,5.24)(74.30,5.26)(74.40,5.28)(74.50,5.30)(74.60,5.32)(74.70,5.34)(74.80,5.36)(74.90,5.38)(75.00,5.40)(75.10,5.42)(75.20,5.44)(75.30,5.46)(75.40,5.48)(75.50,5.50)(75.60,5.52)(75.70,5.54)(75.80,5.56)(75.90,5.58)(76.00,5.60)(76.10,5.62)(76.20,5.64)(76.30,5.66)(76.40,5.68)(76.50,5.70)(76.60,5.72)(76.70,5.74)(76.80,5.76)(76.90,5.78)(77.00,5.80)(77.10,5.82)(77.20,5.84)(77.30,5.86)(77.40,5.88)(77.50,5.90)(77.60,5.92)(77.70,5.94)(77.80,5.96)(77.90,5.98)(78.00,6.00)
        \psline[showpoints=false,linestyle=dashed,dash=2pt 2pt](48.00,0.00)(48.10,3.69)(48.20,4.06)(48.30,4.31)(48.40,4.51)(48.50,4.67)(48.60,4.82)(48.70,4.95)(48.80,5.07)(48.90,5.18)(49.00,5.29)(49.10,5.39)(49.20,5.48)(49.30,5.57)(49.40,5.65)(49.50,5.73)(49.60,5.81)(49.70,5.89)(49.80,5.96)(49.90,6.03)(50.00,6.10)(50.10,6.17)(50.20,6.23)(50.30,6.30)(50.40,6.36)(50.50,6.42)(50.60,6.48)(50.70,6.54)(50.80,6.59)(50.90,6.65)(51.00,6.71)(51.10,6.76)(51.20,6.81)(51.30,6.87)(51.40,6.92)(51.50,6.97)(51.60,7.02)(51.70,7.07)(51.80,7.12)(51.90,7.17)(52.00,7.22)(52.10,7.27)(52.20,7.31)(52.30,7.36)(52.40,7.41)(52.50,7.45)(52.60,7.50)(52.70,7.54)(52.80,7.59)(52.90,7.63)(53.00,7.68)(53.10,7.72)(53.20,7.76)(53.30,7.81)(53.40,7.85)(53.50,7.89)(53.60,7.93)(53.70,7.98)(53.80,8.02)(53.90,8.06)(54.00,8.10)(54.10,8.14)(54.20,8.18)(54.30,8.22)(54.40,8.26)(54.50,8.30)(54.60,8.34)(54.70,8.38)(54.80,8.42)(54.90,8.46)(55.00,8.50)(55.10,8.54)(55.20,8.58)(55.30,8.61)(55.40,8.65)(55.50,8.69)(55.60,8.73)(55.70,8.77)(55.80,8.80)(55.90,8.84)(56.00,8.88)(56.10,8.91)(56.20,8.95)(56.30,8.99)(56.40,9.03)(56.50,9.06)(56.60,9.10)(56.70,9.13)(56.80,9.17)(56.90,9.21)(57.00,9.24)(57.10,9.28)(57.20,9.31)(57.30,9.35)(57.40,9.39)(57.50,9.42)(57.60,9.46)(57.70,9.49)(57.80,9.53)(57.90,9.56)(58.00,9.60)(58.10,9.63)(58.20,9.67)(58.30,9.70)(58.40,9.74)(58.50,9.77)(58.60,9.81)(58.70,9.84)(58.80,9.88)(58.90,9.91)(59.00,9.94)(59.10,9.98)(59.20,10.01)(59.30,10.05)(59.40,10.08)(59.50,10.12)(59.60,10.15)(59.70,10.18)(59.80,10.22)(59.90,10.25)(60.00,10.29)(60.10,10.32)(60.20,10.35)(60.30,10.39)(60.40,10.42)(60.50,10.45)(60.60,10.49)(60.70,10.52)(60.80,10.55)(60.90,10.59)(61.00,10.62)(61.10,10.65)(61.20,10.69)(61.30,10.72)(61.40,10.75)(61.50,10.79)(61.60,10.82)(61.70,10.85)(61.80,10.89)(61.90,10.92)(62.00,10.95)(62.10,10.99)(62.20,11.02)(62.30,11.05)(62.40,11.09)(62.50,11.12)(62.60,11.15)(62.70,11.19)(62.80,11.22)(62.90,11.25)(63.00,11.28)(63.10,11.32)(63.20,11.35)(63.30,11.38)(63.40,11.42)(63.50,11.45)(63.60,11.48)(63.70,11.51)(63.80,11.55)(63.90,11.58)(64.00,11.61)(64.10,11.65)(64.20,11.68)(64.30,11.71)(64.40,11.74)(64.50,11.78)(64.60,11.81)(64.70,11.84)(64.80,11.88)(64.90,11.91)(65.00,11.94)(65.10,11.97)(65.20,12.01)(65.30,12.04)(65.40,12.07)(65.50,12.10)(65.60,12.14)(65.70,12.17)(65.80,12.20)(65.90,12.24)(66.00,12.27)(66.10,12.30)(66.20,12.33)(66.30,12.37)(66.40,12.40)(66.50,12.43)(66.60,12.47)(66.70,12.50)(66.80,12.53)(66.90,12.56)(67.00,12.60)(67.10,12.63)(67.20,12.66)(67.30,12.70)(67.40,12.73)(67.50,12.76)(67.60,12.79)(67.70,12.83)(67.80,12.86)(67.90,12.89)(68.00,12.93)(68.10,12.96)(68.20,12.99)(68.30,13.03)(68.40,13.06)(68.50,13.09)(68.60,13.12)(68.70,13.16)(68.80,13.19)(68.90,13.22)(69.00,13.26)(69.10,13.29)(69.20,13.32)(69.30,13.36)(69.40,13.39)(69.50,13.42)(69.60,13.46)(69.70,13.49)(69.80,13.52)(69.90,13.56)(70.00,13.59)(70.10,13.62)(70.20,13.66)(70.30,13.69)(70.40,13.72)(70.50,13.76)(70.60,13.79)(70.70,13.82)(70.80,13.86)(70.90,13.89)(71.00,13.93)(71.10,13.96)(71.20,13.99)(71.30,14.03)(71.40,14.06)(71.50,14.09)(71.60,14.13)(71.70,14.16)(71.80,14.20)(71.90,14.23)(72.00,14.26)(72.10,14.30)(72.20,14.33)(72.30,14.37)(72.40,14.40)(72.50,14.43)(72.60,14.47)(72.70,14.50)(72.80,14.54)(72.90,14.57)(73.00,14.60)(73.10,14.64)(73.20,14.67)(73.30,14.71)(73.40,14.74)(73.50,14.78)(73.60,14.81)(73.70,14.85)(73.80,14.88)(73.90,14.91)(74.00,14.95)(74.10,14.98)(74.20,15.02)(74.30,15.05)(74.40,15.09)(74.50,15.12)(74.60,15.16)(74.70,15.19)(74.80,15.23)(74.90,15.26)(75.00,15.30)(75.10,15.33)(75.20,15.37)(75.30,15.40)(75.40,15.44)(75.50,15.47)(75.60,15.51)(75.70,15.54)(75.80,15.58)(75.90,15.62)(76.00,15.65)(76.10,15.69)(76.20,15.72)(76.30,15.76)(76.40,15.79)(76.50,15.83)(76.60,15.87)(76.70,15.90)(76.80,15.94)(76.90,15.97)(77.00,16.01)(77.10,16.04)(77.20,16.08)(77.30,16.12)(77.40,16.15)(77.50,16.19)(77.60,16.23)(77.70,16.26)(77.80,16.30)(77.90,16.34)(78.00,16.37)
    \end{pspicture}
    \vspace*{-0.2cm}
    \end{center}
    \caption{Comparison of JCC bounds for WSS (solid line) and NRM (dashed line) disturbances, as
    defined in Equation \eqref{Equ:CompJCCBound}.
    \label{Fig:JCCBounds}}
\end{figure}
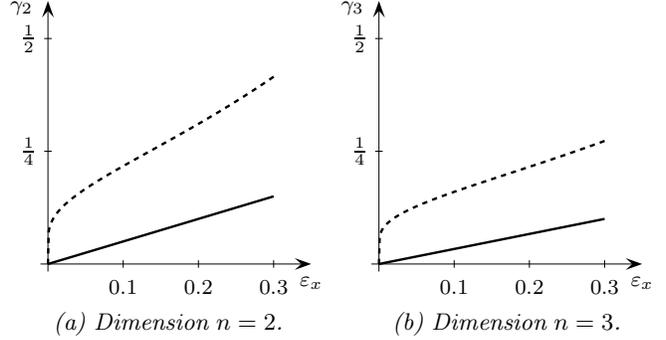


\section{Controller Synthesis Problem}\label{Sec:Synthesis}

In this section, the controller synthesis problem (CSP) is assembled. In particular, two main questions on the CSP must be addressed. The first is about its \emph{feasibility}: what probability levels $\epx$ and/or $\epu$ of the chance constraints \eqref{Equ:ChanceConstr} are achievable? The second is about its \emph{optimality}: provided that the CSP is feasible, which controller minimizes the objective function \eqref{Equ:CostFunction} while remaining constraint admissible?

\subsection{Feasibility of the CSP}\label{Sec:Feasibility}

From \eqref{Equ:StationaryVar}, $\bX(K)\succ 0$ is lower bounded by $\bX(K)\succeq W$ independent of the choice of $K$. Therefore the probability level $\epx$ or $\epu$ that is achievable for a given constraint set $\BX$ or $\BU$ is generally not arbitrarily small. It is therefore important to obtain lower bounds on the probability levels in the CSP that guarantee feasibility of the problem.

For simplicity, consider first the CSP with only a single SCC or JCC on the state (or analogously, the input). Then one can solve for the minimum probability level $\bepx>0$ (or $\bepu>0$), as shown below.

\begin{theorem}[Minimum Probability Level]\label{The:IndMinEp}
    Con-sider a single chance constraint on the state \eqref{Equ:StateCon}. The solution $\bepx$
    to
	\begin{subequations}\label{Equ:IndMinEp}\begin{align}
    &\hspace*{-0.1cm}\min_{X,Y,\epx}\quad \epx\\
    &\st\quad \begin{bmatrix} X-W & (AX+BY) \\ (AX+BY)\tp & X \end{bmatrix} \succeq 0\ec\\
    &\pst\quad  \text{\eqref{Equ:SCCState}, \eqref{Equ:JCCWSSState}, or \eqref{Equ:JCCNRMState}}
	\end{align}\end{subequations}
	represents a bound on the lowest feasible probability level for this chance constraint. It is
	achieved by $K=YX^{-1}$, where $X$ and $Y$ are the solution of \eqref{Equ:IndMinEp}.\\
	Analogously, the minimum probability level $\bepu$ can be found for single chance constraint on
	the input \eqref{Equ:InputCon}: It is the solution to \eqref{Equ:IndMinEp} with $\epx$ replaced
	by $\epu$, and (\ref{Equ:IndMinEp}c) being either \eqref{Equ:SCCInput}, \eqref{Equ:JCCWSSInput},
	or \eqref{Equ:JCCNRMInput}.
\end{theorem}

\begin{proof}
    The result is a straightforward application of Lemmas \ref{The:SCCState}, \ref{The:SCCInput},
    \ref{The:JCCState} and \ref{The:JCCInput}.
\end{proof}

\begin{remark}[Computation]\label{Rem:Computation}
	In the case of NRM disturbances, \eqref{Equ:IndMinEp} cannot be solved directly for $\bepx$ or
	$\bepu$. For an SCC on the state or input, one should minimize $\alpha$ or $\beta$ instead; for
	a JCC on the state or input, one should minimize $\gamma_{n}$ instead. The minimal probability
	level $\bepx$ or $\bepu$ is then found by inverting the corresponding (monotonic) relationship
	from Table \ref{Tab:SCC} or Equation \eqref{Equ:CompJCCBound}.
\end{remark}

The extension to a CSP with multiple constraints is straightforward. To this end, the achievable probability levels of multiple constraints generally conflict with each other. For example, lower input usage by a lower probability level $\epu$ may increase the achievable lower bound on the probability level $\epx$ for a state constraint. 

In order to find a feasible combination of probability levels, the constraints should be listed according to their priority. Then their probability levels are fixed in the order of that list. Each constraint whose probability level has been fixed is added as an additional LMI to \eqref{Equ:IndMinEp}.

\subsection{Optimality of the CSP}\label{Sec:Optimality}

\begin{theorem}[Optimal Unconstrained Feedback]\label{The:OptFeedback}
	The optimal unconstrained feedback gain $K_{\lqr}$ can be computed as $K_{\lqr}:=YX^{-1}$ from
	the solution $X,Y$ to
	\begin{subequations}\label{Equ:OptFeedback}\begin{align}
		&\hspace*{-0.1cm}\min_{X,Y,P}\quad \Tr\bigl(QX\bigr) + \Tr\bigl(P\bigr)\\
    	&\st\quad \begin{bmatrix} P & (R^{1/2}Y) \\ (R^{1/2}Y)\tp & X \end{bmatrix} \succeq 0\ec\\
		&\pst\quad \begin{bmatrix} X-W & AX+BY \\ (AX+BY)\tp & X \end{bmatrix} \succeq 0 \ec
	\end{align}\end{subequations}
	where $P\in\BR^{n\times n}$ is an auxiliary positive semidefinite matrix.\\
	Moreover, the solution $X$ to this problem is exactly the stationary covariance of the state 
	with LQR feedback, $X=\bX(K_{\lqr})$.
\end{theorem}

\begin{proof}
	The cost function \eqref{Equ:CostFunction} can be reformulated into (\ref{Equ:OptFeedback}a,b)
	by algebraic manipulations:
	\begin{align}\label{Equ:OptFeedbackEqu1}
		\E\bigl[x\tp Qx \!&+\! u\tp Ru\bigr]\!= \nonumber\\
	  	&=\E\bigl[\Tr\bigl(x\tp Qx\bigr)\bigr] \!+\! 
	                                          \E\bigl[\Tr\bigl(x\tp K\tp RKx\bigr)\bigr] \nonumber\\
		&=\Tr\bigl(Q\bX\bigr) + \Tr\bigl(K\tp R K \bX\bigr)\nonumber\\
		&\leq\Tr\bigl(QX\bigr) + \Tr\bigl(K\tp RKX\bigr)\nonumber\\
		&=\Tr\bigl(QX\bigr) + \Tr\bigl(R^{1/2}KXX^{-1}X\tp K\tp R^{1/2\textrm{T}}\bigr)\nonumber\\
		&=\Tr\bigl(QX\bigr) + \Tr\bigl(R^{1/2}YX^{-1}Y\tp R^{1/2\textrm{T}}\bigr)\,.
	\end{align}
	Minimizing the second term in the above is equivalent to solving
	\begin{equation*}\begin{split}
		\min_{P}\quad & \Tr\bigl(P)\\
		\text{s.t.}\quad&\begin{bmatrix} P &(R^{1/2}Y)\\ (R^{1/2}Y)\tp &X\end{bmatrix}\succeq 0\ec\\
	\end{split}\end{equation*}
	for an auxiliary variable $P\succeq 0$ \cite[Sec.\,2.1]{BoydEtAl:1994}. This proves the first
	part of the theorem, provided that $X=\bX$.
	
	For the sake of a contradiction, suppose that $X\neq\bX$. Since the solution $X$ is feasible for
	(\ref{Equ:OptFeedback}c), it satisfies $X\succeq\bX$ by Lemma \ref{The:Stationarity}. If $X-\bX
	\succeq 0$, then the inequality in \eqref{Equ:OptFeedbackEqu1} becomes a strict inequality,
	because $Q\succ 0$. However, this means that the cost function value can be improved by choosing
	the feasible solution $X=\bX$.
\end{proof}

Let the probability levels of all chance constraints be fixed such that set of feasible controllers, according to \eqref{Equ:IndMinEp}, is non-empty. The linear control law that minimizes the expected quadratic cost \eqref{Equ:CostFunction} subject to the chance constraints can hence be computed by including the respective constraints in the CSP:
\vspace*{-1ex}
\begin{itemize}
	\item 
	in the WSS case \eqref{Equ:SCCState}, \eqref{Equ:SCCInput}, 
	\eqref{Equ:JCCWSSState}, \eqref{Equ:JCCWSSInput},  \\
	\item 
	in the NRM case \eqref{Equ:SCCState}, \eqref{Equ:SCCInput},
	\eqref{Equ:JCCNRMState}, \eqref{Equ:JCCNRMInput}.
\end{itemize}
\vspace*{-1ex}
The chance constrained controller can then be computed as $K_{\cc}:=YX^{-1}$, where $X$ and $Y$ are the solution to the resulting constrained CSP. As in the unconstrained case, the stationary covariance will satisfy $X=\bX(K_{\cc})$ by the same argument as in Theorem \ref{The:OptFeedback}.

\section{Output Feedback}\label{Sec:OutputFeedback}

So far, it has been assumed that the state of \eqref{Equ:System} can be measured for the purpose of state feedback. For the case of output feedback, all of the preceding results can be extended by inclusion of an optimal \emph{state estimator}:
\begin{equation}\label{Equ:OutFeedback}
    x_{t}=:\hx_{t}+e_{t}\ec\qquad u_{t}=K\hx_{t}\ec
\end{equation}
where $\hx_{t}$ represents the \emph{state estimate} of $x_{t}$ at time $t$ and $e_{t}$ the \emph{estimation error}. 

\subsection{Kalman Filter (KF)}\label{Sec:KalmanFilter}

For a generic state estimator, the dynamics of the estimates $\{\hx_{t}\}_{t\in\BN}$ are given by
\begin{subequations}\label{Equ:KalmanFilter}\begin{align}
    &\hx_{t+1}=A\hx_{t}+Bu_{t}+L(y_{t}-\hy_{t})\ec\\
    &\hy_{t}=C\hx_{t}\ec
\end{align}\end{subequations}
where $L\in\BR^{n\times p}$ is the \emph{estimator gain}. The particular estimator gain of the \emph{Kalman Filter} (KF) can be computed directly from the system data $A,B,C$ and the covariance matrices $V,W$. 

The KF is known to provide state estimates that are unbiased and minimum variance:
\begin{subequations}\label{Equ:KFProp}\begin{align}
    &\E\bigl[e_{t}\hx_{t}\tp\bigr]=0\quad\fa t\in\BN\ec\\
    &\min_{L}\,\upsilon\tp \E\bigl[e_{t}e_{t}\tp\bigr]\upsilon
     \quad\fa\upsilon\in\BRn,\enspace t\in\BN\ef
\end{align}\end{subequations}
Equation (\ref{Equ:KFProp}b) means that the KF has the minimum variance of all linear estimators, with respect to the semidefinite partial ordering. If the disturbances are NRM, the KF has the minimum variance of all possible estimators \cite[Cha.\,5]{Simon:2006}, \cite[Sec.\,9.3]{Ludyk2:1995}. Analogous to the \emph{separation principle} in classical LQG control, for the chance-constrained LQR the KF gain $L$ can therefore be assumed independently of the state feedback gain $K$.

Compared to state feedback, the chance constraints must be adjusted for the additional variance introduced by the estimation error. This is shown in the following. 

\subsection{Stationarity Conditions}\label{Sec:KFStationarity}

In stationary operation with a stable estimator gain $L$, the state estimate and the estimation error are stationary variables $\hx$ and $e$, whose first two moments are
\begin{subequations}\begin{align*}
    &\E\bigl[\hx\bigr]=0\ec\qquad\E\bigl[e\bigr]=0\ec\\
    &\bS:=\E\bigl[\hx\hx\tp\bigr]\ec\qquad E:=\E\bigl[ee\tp\bigr]\ef
\end{align*}\end{subequations}
The stationary error variance $E$ is obtained, along with the estimator gain $L$, from the design of the Kalman Filter. For the case of output feedback, Lemmas \ref{The:StationaryVar}, \ref{The:Stationarity} are now replaced with the following results.

\begin{lemma}[Stationarity Condition]\label{The:KFStationarity}
	Let $K$ be stabilizing for $(A,B)$ and $\{w_{t}\}_{t\in\BN}$ be WSS or NRM. The covariance of
	the KF state estimate $\{\hx_{t}\}_{t\in\BN}$ converges to a unique stationary value 
	$\bS\succeq 0$ that satisfies the discrete time Lyapunov equation
    \begin{multline}\label{Equ:KFStationaryEq}
    	\bS = (A+BK)\bS(A+BK)\tp +\\ + (LC)E(LC)\tp + LVL\tp\ef    
    \end{multline}
    Moreover, the covariance of the state $\{x_{t}\}_{t\in\BN}$ converges to a unique stationary
	value $\bX\succ 0$ that satisfies
	\begin{equation}\label{Equ:KFStationaryState}
    	\bX = \bS + E\ef
    \end{equation}
\end{lemma}

\begin{proof}
    Substitute (\ref{Equ:System}b), (\ref{Equ:KalmanFilter}b), \eqref{Equ:OutFeedback} into
    (\ref{Equ:KalmanFilter}a) to obtain the dynamics of the state estimate:
    \begin{align}\label{Equ:Estimate}
        \hx_{t+1}&=A\hx_{t}+Bu_{t}+L(y_{t}-\hy_{t})\nonumber\\
                 &=(A+BK)\hx_{t}+LCe_{t}+Lv_{t}\ef
    \end{align}
    Equation \eqref{Equ:KFStationaryEq} now follows analogously to Lemma \ref{The:StationaryVar},
    given that $v_{t}$ is independent of $\hx_{t}$ and $e_{t}$ by assumption, and that $\hx_{t}$ and
    $e_{t}$ are uncorrelated by (\ref{Equ:KFProp}a).
    
    Equation \eqref{Equ:KFStationaryState} follows from \eqref{Equ:OutFeedback},
    \begin{equation}\label{Equ:KFStatProof1}
        \bX=\E\bigl[xx\tp\bigr]=\E\bigl[(\hx+e)(\hx+e)\tp\bigr]=S+E\ef
    \end{equation}
    using the fact that $\hx_{t}$ and $e_{t}$ are uncorrelated (\ref{Equ:KFProp}a).
\end{proof}

\begin{lemma}[Stationarity Condition]\label{The:Stationarity} If $K\in\BR^{m\times n}$, $S\in
	\BR^{n\times n}$, and $X\in\BR^{n\times n}$ are chosen such that
    \begin{align*}
    &\begin{bmatrix} S-(LC)E(LC)\tp-LVL\tp & (A+BK)S \\ S\tp(A+BK)\tp & S\end{bmatrix}\succeq 0\ec\\
    &X\succeq S + E\ec
    \end{align*}
    then $(A+BK)$ is stable and $S\succeq\bS(K)$, $X\succeq\bX(K)$.
\end{lemma}

\begin{proof}
	The stability of $(A+BK)$ and $S\succeq\bS$ follow analogously to the proof of Lemma 
	\ref{The:Stationarity}. 
	The fact that $X\succeq\bX(K)$ follows from \eqref{Equ:KFStationaryState} and $S\succeq\bS$.
\end{proof}

For the case of output feedback, several changes to the CSP with state feedback must be made. They are summarized in the following remark.

\begin{remark}[CSP for Output Feedback]\label{Rem:KFDecVar}
	(a) The decision variables of the CSP with output feedback are $X$, $S$, and
	\begin{equation}\label{Equ:KFTransform}
		Z:=KS \qquad\Longrightarrow\qquad K=ZS^{-1}\ec
	\end{equation}
	as compared to $X$ and $Y:=KX$ for the CSP with state feedback.
	(b) In all input constraint reformulations \eqref{Equ:SCCInput}, \eqref{Equ:JCCWSSInput},
	\eqref{Equ:JCCNRMInput}, $X$ must be replaced with $S$ and $Y$ with $Z$.
	(c) In the cost function reformulation, $X$ must be replaced with $S$ and $Y$ with $Z$ in
	(\ref{Equ:OptFeedback}b).
\end{remark}

For the same argument as in the proof of Theorem \ref{The:OptFeedback}, the optimal covariances of the CSP are not chosen any bigger than necessary; \ie $S=\bS(K_{\lqr})$, $X=\bS(K_{\lqr})$ in the unconstrained case and $S=\bS(K_{\cc})$, $X=\bS(K_{\cc})$ in the constrained case.

\section{Example}\label{Sec:Example}

Consider the model of a spinning satellite \cite[Sec.\,10.2]{FrPoEm:2009}, as depicted in Figure  \ref{Fig:SatelliteModel}. It consists of two rotating masses: the first mass with inertia $J_{1}=1$ represents the satellite body with thrusters, and the second mass with inertia $J_{2}=0.1$ carries the instruments. 

The two masses are connected with a boom of low stiffness $k=0.02$ and damping $b=0.0001$. Disturbing torques act on both masses, while a torque on the first mass can be applied for control of the satellite. Time is discretized using a sampling period of $\Delta t=0.1$. The objective is to maintain the instruments in a stable position, while using as little thrust as possible.

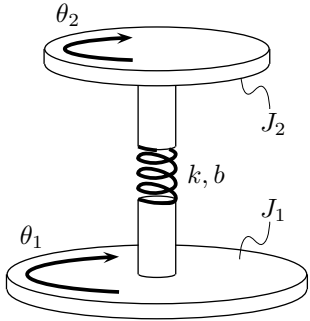
\begin{figure}[H]
	\begin{center}
	\begin{pspicture}(0,0)(50,43)
        \mel
        \psellipticarc[fillcolor=white](25,35)(15,3){-180}{0}
        \psellipse[fillcolor=white](25,37)(15,3)
        \psline(10,37)(10,35)\psline(40,37)(40,35)
        \psellipticarc[fillcolor=white](25,5)(20,4){-180}{0}
        \psellipticarc[fillcolor=white](25,7)(20,4){-238}{58}
        \psline(5,7)(5,5)\psline(45,7)(45,5)
        \psellipticarc[fillcolor=white](25,7)(2.5,0.5){-180}{0}
        \psline(22.5,7)(22.5,17)\psline(22.5,24)(22.5,32.1)
        \psline(27.5,7)(27.5,17)\psline(27.5,24)(27.5,32.1)
        \psellipticarc[fillcolor=white](25,24)(2.5,0.5){-180}{0}
        \psellipticarc[fillcolor=white](25,17)(2.5,0.5){-180}{180}
        \bol
        \pscurve(22.5,17)(25,16.5)(27.5,17)
        (25.6,18.7)(22.5,18.5)(23,18.1)(26,18)(27.5,19)
        (25.6,20.7)(22.5,20.5)(23,20.1)(26,20)(27.5,21)
        (25.6,22.7)(22.5,22.5)(23,22.1)(26,22)(27.5,23)
        (26.63,23.65)
        \pscurve(22.5,24)(23.5,23.7)(25,23.5)
        \lal
        \pscurve(36,9)(37,12)(39,11)(40,14)
        \rput[bc](40.5,15.5){$J_{1}$}
        \pscurve(36.2,33.1)(37,31)(39,32)(40,29)
		\rput[tc](40.5,27){$J_{2}$}
        \rput[ml](29,20){$k,b$}
        \bol
        \psellipticarc{<-}(25,7)(17.5,2.7){155}{205}
        \rput[br](10,10.5){$\theta_{1}$}
        \psellipticarc{<-}(25,37)(12.5,1.8){155}{205}
        \rput[br](14.8,40){$\theta_{2}$}
    \end{pspicture}
    \end{center}
    \caption{Model of a spinning satellite.\label{Fig:SatelliteModel}}
\end{figure}

Defining the state of the satellite as $x=[\theta_{2}\:\dot{\theta}_{2}\:\theta_{1}\:\dot{\theta}_{1}]\tp$ and the output $y=[\theta_{2}\:\theta_{1}]\tp$ gives the following system matrices:
\begin{subequations}\label{Equ:Ex2System}\begin{align*}
    &A = \begin{bmatrix} 0.993  &0.100    & 0.008   &0.000\\
                        -0.150  &0.992    & 0.150   &0.008\\
                         0.002  &0.000    & 0.999   &0.100\\
                         0.030  &0.002    &-0.030   &0.999\end{bmatrix},\quad
	B = \begin{bmatrix} 0.000\\ 0.000\\ 0.001\\ 0.010\\ \end{bmatrix},\\
	&\qquad\qquad C = \begin{bmatrix} 1.000 &0.000 &0.000& 0.000 \\
	                                  0.000 &0.000  &1.000 &0.000 \end{bmatrix}.
\end{align*}\end{subequations}
In each step, the disturbance torques are normally distributed with variance $0.1$, and the measurement noise is normally distributed with variance $0.05$. The quadratic cost matrices
\begin{equation}\label{Equ:Ex2Cost}
	Q = 0.1\cdot I\ec\quad R = 1\ec
\end{equation}
reflect the expensiveness of thruster use in space.

First, consider a bound on the control torque:
\begin{equation}\label{Equ:ExInput}
	\Pb\bigl[|u|\leq 1\bigr]\geq 1-\epu\ec\quad \epu=10\%\ef
\end{equation}
The lower bound on $\epu$, according to Theorem \ref{The:IndMinEp}, is $\bepu=0\%$ because the system is open-loop stable. 

The optimal chance-constrained controller $K_{\cc}$ shall be compared to the optimal unconstrained controller $K_{\lqr}$, based on a numerical simulation of the two closed-loop systems for $10^6$ time steps. In particular, the \emph{empirical violations} observed in this simulation are $\tep_{u,\cc}=9.18\%$ and $\tep_{u,\lqr}=39.83\%$, respectively. Notice that $K_{\cc}$ keeps the chance constraints quite exactly, that is with little conservatism.

Second, for keeping the instruments steady, consider a constraint on the position of the second mass:
\begin{equation}\label{Equ:ExState1}
	\Pb\bigl[|\theta_{2}|\leq 5\bigr]\geq 1-\epx\ec\quad \epx=10\%\ef
\end{equation}
In fact, the lower bound for $\epx$ is $\bep_{x}=0.00\%$. Here the empirical violations amount to $\tep_{x,\cc}=10.08\%$ for the chance-constrained LQR, compared to $\tep_{x,\lqr}=28.59\%$ for the unconstrained LQR. Again, notice that there is little conservatism for this constraint.

Finally, consider a constraint jointly on the position and the velocity of the second mass:
\begin{equation}\label{Equ:ExState2}
	\Pb\bigl[\theta_{2}^2+\Delta t\cdot\dot{\theta}_{2}^2\leq 5\bigr]\geq 1-\epx\ec\quad\epx=10\%\ef
\end{equation}
Here \eqref{Equ:ExState2} can be put in the form of \eqref{Equ:StateJCC} by choosing $G^{-1}$ as the diagonal matrix with the entries $1$, $\Delta t$, and two artificial, very small numbers.

The lower bound for $\epx\geq\bep_{x}=14.11\%$ lies above the desired constraint level of $10\%$. However, if a constrained LQR controller is designed, the empirical violations amount to $\tep_{x,\cc}=10.04\%$ and $\tep_{x,\lqr}=44.04\%$, respectively. While there is little conservatism in this case, in general there may be a discrepancy between the violation bound and the empirically observed violations.

\begin{remark}[Conservatism]
	Conservatism in the chance-constrained controller design may appear for three reasons: 
	(a) If the disturbance is WSS, the Chebyshev bounds of Lemmas \ref{The:SCCState}, 
	\ref{The:SCCInput}, \ref{The:JCCState}, \ref{The:JCCInput} not necessarily tight for the
	stationary distribution.
	(b) The tail bounds on multi-variate distributions in Lemmas \ref{The:JCCState}, 
	\ref{The:JCCInput} are not generally tight for both NRM and WSS disturbances.
	(c) There may not exist a linear controller $K$ that produces stationary covariances that match
	the bounds of Lemmas \ref{The:JCCState}, \ref{The:JCCInput} exactly, even if they were tight.
\end{remark}

\section{Conclusion}\label{Sec:Conclusion}

In this paper, a design approach for a LQR was presented that is optimal with respect to a quadratic target function, while respecting chance constraints on the input and/or state in closed-loop operation. This problem can be reformulated in terms of various LMIs, leading to a linear SDP that can be solved efficiently by standard toolboxes, such as CVX \cite{Grant:2012}.

The approach offers three key advantages to existing methods. First, its focus on the stationary operation of the system, as opposed to optimizing a transient response. Second, the controller synthesis problem is tractable in offline optimization, and then a simple linear controller has to be implemented online. Third, the approach can be extended to output feedback, in which case the theory covers the controller-observer combination.

Note that, compared to the standard LQG, the same performance may be achieved by a proper choices of $Q$ and $R$; however, the tuning procedure may be tedious in practice and does not bring about an optimality guarantee. 

\balance
\bibliographystyle{plain}
\bibliography{contrbib,mathbib,engbib,paul}

\end{document}